\documentclass[11pt]{amsart}
\usepackage{amssymb}
\usepackage{times}
\usepackage{url}
\usepackage{hyperref}

\newtheorem{theorem}{Theorem}[section]
\newtheorem{lemma}[theorem]{Lemma}
\newtheorem{proposition}[theorem]{Proposition}
\newtheorem{corollary}[theorem]{Corollary}
\theoremstyle{remark}
\newtheorem{remark}[theorem]{Remark}

\renewcommand{\P}{\mathbb{P}}
\renewcommand{\mod}{\hspace{4pt}\mathrm{mod}\hspace{4pt}}

\newcommand{\F}{\mathbb{F}}
\newcommand{\Z}{\mathbb{Z}}
\newcommand{\Q}{\mathbb{Q}}

\newcommand{\GL}{\mathrm{GL}}
\newcommand{\SL}{\mathrm{SL}}

\newcommand{\slnq}{\mathrm{SL}(n,\Q)}
\newcommand{\slnz}{\mathrm{SL}(n,\Z)}

\newcommand{\slnm}{\mathrm{SL}(n,\Z_m)}
\newcommand{\glnm}{\mathrm{GL}(n,\Z_m)}
\newcommand{\glnz}{\mathrm{GL}(n,\Z)}
\newcommand{\glnq}{\mathrm{GL}(n,\Q)}
\newcommand{\glnp}{\mathrm{GL}(n,p)}

\newcommand{\fis}{\leq_f \!}
\newcommand{\st}{\hspace{3pt} | \hspace{3pt}}

\newcommand{\gpess}{\langle S \hspace{1pt} \rangle}

\newcommand{\abk}{\allowbreak}

\begin{document}

\title[Algorithms for arithmetic groups]{Algorithms for arithmetic
groups with the congruence subgroup property}

\dedicatory{Dedicated to the  memory of  \'Akos Seress}

\begin{abstract}
We develop practical techniques to compute with arithmetic groups
$H\leq \slnq$ for $n>2$. Our approach relies on constructing a
principal congruence subgroup in $H$. Problems solved include
testing membership in $H$, analyzing the subnormal structure of
$H$, and the orbit-stabilizer problem for $H$. Effective
computation with subgroups of $\GL(n,\mathbb{Z}_m)$ is vital to
this work. All algorithms have been implemented in {\sf GAP}.
\end{abstract}

\author{A.~S. Detinko}
\author{D.~L. Flannery}
\author{A.~Hulpke}

\maketitle

In \cite{Tits,FRecog,SF} we established methods for computing with
finitely generated linear groups over an infinite field, based on
the use of congruence homomorphisms. These have been applied to
test virtual solvability and answer questions about
solvable-by-finite (SF) linear groups.

Computing with finitely generated linear groups that are not SF is
a largely unexplored topic. These
groups comprise a wide class in which certain algorithmic problems
are undecidable \cite[Section~3]{Survey}. We might be more confident
of progress if we restrict ourselves to arithmetic subgroups of
linear algebraic groups. Decision problems for such groups were
investigated by Grunewald and Segal~\cite{GSI}; see also
\cite{DeGrDetF}. We note renewed activity focussed on deciding
arithmeticity~\cite{Sarnak}.

This paper is a revision of \cite{ArithmPublished}, which
provides a starting point for computation with semisimple
arithmetic groups that have the congruence subgroup property
(CSP). A prominent example is $\Gamma_n=\abk \slnz$, $n\geq\abk
3$. Recall that $H\leq \SL(n,\Q)$ is arithmetic if $\Gamma_n\cap
H$ has finite index in both $H$ and $\Gamma_n$ (in particular,
finite index subgroups of $\Gamma_n$ are arithmetic). Each
arithmetic group $H\leq \slnq$ contains a principal congruence
subgroup $\Gamma_{n,m}$ for some $m$, namely the kernel of the
congruence homomorphism $\Gamma_n \rightarrow \allowbreak \slnm$
induced by natural surjection $\Z \rightarrow \Z_m:=
\Z/m\Z$~\cite{BLS, Mennicke}. So if we know that $\Gamma_{n,m}\leq
H$ then we can transfer much of the computing to $\slnm$, for
which efficient machinery is available~\cite{Hulpke2013}. We give
a method to construct $\Gamma_{n,m}$ in $H$. This implies
that membership testing and other fundamental problems
are decidable.

We pay special attention to subnormality and the orbit-stabilizer
problem. Aside from their computational importance, these were the
earliest questions considered for arithmetic groups. The study of
subnormal subgroups of $\Gamma_n$ originated in the late 19th
century and led up to formulation of the Congruence Subgroup
Problem. In turn, the solution of that problem used knowledge of
$\Gamma_n$-orbits in $\Q^n$ \cite[\S 17] {Humphreys}.

The paper is organized as follows. Section~\ref{Background}
covers background on arithmetic groups: basic facts; material
about principal congruence subgroups (their generating sets, 
maximality); and subnormal structure.
Section~\ref{AlgorithmsforgroupsoverZm} details relevant theory of
matrix groups over $\Z_m$ and computing in $\glnm$. In
Section~\ref{AlgorithmsArithmetica} we give a suite of algorithms
for arithmetic groups in $\Gamma_n$. After verifying decidability,
we describe computing a maximal principal congruence subgroup;
membership testing; and aspects of subnormality, e.g., testing
whether an arithmetic group $H\leq \abk \Gamma_{n}$ is subnormal
or normal, and constructing the normal closure of a subgroup of
$\Gamma_n$. In Section~\ref{OSProblem} we solve the
orbit-stabilizer problem for arithmetic groups in $\Gamma_n$
acting on $\Q^n$. Our solution draws on a comprehensive
description of orbits and stabilizers for a principal
congruence subgroup acting on $\Z^n$.
Section~\ref{IntertwiningArithmeticGroupstoSLnZ} shows how to
extend results from $\Gamma_n$ to $\slnq$. Finally, we examine the
performance of our {\sf GAP} \cite{GAP} implementation of the
algorithms.

We remark that the scope of this paper may be widened to other
groups with the CSP, such as $\mathrm{Sp}(2m,\mathcal{O}_\P)$ or
$\SL(n,\mathcal{O}_\P)$ for $m\geq 2$ and $n>2$, where
$\mathcal{O}_\P$ is the ring of integers of a number field $\P$
that is not totally imaginary~\cite{BLS}.

\section{Arithmetic subgroups of $\SL(n,\Q)$: background}
\label{Background}

\subsection{Preliminaries}\label{Prelim}

Let $R$ be a commutative ring with $1$, and $I\subseteq R$ be an
ideal. The natural surjection $R\rightarrow\abk R/I$ induces a
congruence homomorphism $\varphi_I:\mathrm{Mat}(n,R)\rightarrow
\mathrm{Mat}(n,R/I)$. Let $G_n =\abk  \GL(n,R)$ and $\Gamma_n =
\abk \SL(n,R)$. The kernel of $\varphi_I$ on $\Gamma_n$ or $G_n$
is a \emph{principal congruence subgroup} (PCS) \emph{of level}
$I$. Such a subgroup of $\Gamma_n$ will be denoted $\Gamma_{n,I}$.
We set $\Gamma_{n,R}= \abk \Gamma_n$. If $R = \abk \Z$ then $R/I =
\Z_m$ for some non-negative integer $m$, and the subscript `$I$'
is replaced by `$m$'.

For computational purposes, $\Gamma_n$ and $G_n$ should be
finitely generated, and proper quotients of $R$ should be finite.
The latter is true if $R= \mathcal{O}_\P$ or $R$ is the univariate
polynomial ring $\F_q[\mathrm{x}]$ over the finite field $\F_q$ of
size $q$. These are two major types of ambient ring $R$
encountered when computing with finitely generated linear groups.

Define $t_{ij}(a)=1_n+e_{ij}(a)$, where $e_{ij}(a)\in
\mathrm{Mat}(n,R)$ has $a$ in position $(i,j)$ and zeros
everywhere else. The matrices $t_{ij}(a)$ for distinct $i$, $j$
are \emph{transvections}. The subgroup
\[
E_{n,I}=\langle t_{ij}(a) :\abk a \in I, \, 1\leq i, j \leq n, \,
i\neq j \rangle
\]
of $\Gamma_{n,I}$ is the \emph{elementary group of level $I$}. We
write $e_{ij}$, $t_{ij}$, $E_n$ for $e_{ij}(1)$, $t_{ij}(1)$,
$E_{n,R}$ respectively.
\begin{lemma}\label{tijRelations} \
\begin{itemize}
\item[{\rm (i)}] For all $i\neq j$, $[t_{ij}(a), t_{ji}(b)] =
1_n+e_{ij}(a^2 b) - e_{ji}(ab^2) + e_{ii}(a b + a^2b^2) -e_{jj}(a
b)$. \item[{\rm (ii)}] If $\, i,j,k$ are pairwise distinct then
$[t_{ij}(a), t_{jk}(b)] = t_{ik}(a b)$ and $[t_{ij}(a), t_{ki}(b)]
= t_{kj}(-a b)$. \item[{\rm (iii)}] If $\, i\neq l$ and $j\neq k$
then $t_{ij}(a)$ commutes with $t_{kl}(b)$.
\end{itemize}
\end{lemma}

\begin{proposition} \label{GenRingElvsGamma}
In each of the following situations, $\Gamma_n=E_n$: {\rm (i)}
$n\geq 2$ and $R$ is Euclidean or semi-local; {\rm (ii)} $n\geq 3$
and $R$ is a Hasse domain of a global field.
\end{proposition}
\begin{proof}
See \cite[4.3.9, pp.~172--173]{HahnOMeara}.
\end{proof}
\begin{remark}
$\mathcal{O}_\P$ is a Hasse domain of a global field,
$\F_q[\mathrm{x}]$ is Euclidean, and $\Z_m$ is semi-local.
\end{remark}

Proposition~\ref{GenRingElvsGamma} implies that $\varphi_m$ maps
$\slnz$ onto $\slnm$. However, $\varphi_I: \GL(n,R)\rightarrow
\abk \GL(n,R/I)$ may not be surjective.
\begin{proposition} \label{GenRingElvsGamma2}
Let $R=\mathcal{O}_\P$ or $\F_q[\mathrm{x}]$. If $n>2$ or $R =
\mathcal{O}_\P$ then $E_n$, $\Gamma_n$, and $G_n$ are finitely
generated. None of the groups $E_2$, $\Gamma_2$, or $G_2$ is
finitely generated when $R= \F_q[\mathrm{x}]$.
\end{proposition}
\begin{proof}
If $n\geq 3$ then $\Gamma_n = E_n$ is finitely generated by
\cite[4.3.11, p.~174]{HahnOMeara}; hence so too is $G_n$, by
\cite[1.2.17, p.~29]{HahnOMeara} and Dirichlet's unit theorem. See
\cite[4.3.16, p.~175]{HahnOMeara} and subsequent comments for the
remaining claims.
\end{proof}

The notation $A\fis B$ means that $A$ is of finite index in the
group $B$. For $n\geq 3$, $\Gamma_n=\abk \slnz$  has the
\emph{congruence subgroup property}: $H\fis \allowbreak \Gamma_n$
is equivalent to $H$ containing some $\Gamma_{n,m}$ \cite{BLS,
Mennicke}. On the other hand, $\Gamma_2$ does not have the CSP
\cite[\S 1.1]{Sury}.

\subsection{Generators of congruence subgroups}
\label{CongSubgroupGens}

Let $R=\Z$. We first discuss generating sets for $G_n$ and
$\Gamma_n$, and thus for their homomorphic images $\overline{G}_n
= \abk \GL(n,\Z_m)$, $\overline{\Gamma}_n = \abk \slnm$.

By Lemma~\ref{tijRelations}~(ii), the transvections
$t_{12},\ldots, t_{1n}, t_{21}, \ldots ,t_{n1}$ constitute a
generating set for $\Gamma_n=\abk E_n$.  In fact $\Gamma_n$ has a
generating set of minimal size $2$: $t_{12}$ and
\[
{\footnotesize \left( \renewcommand{\arraycolsep}{.015cm}
\begin{array}{cc} 0 & 1_{n-1}  \\
\vspace*{-3mm}&\\ (-1)^{n-1} & 0 \end{array}\right)};
\]
see \cite[p.~107]{Newman}. Adding the diagonal matrix
$\mathrm{diag}(-1, 1, \dots , 1)$ produces a generating set for
$G_n$ of size $3$ (better still, it is known that
$G_n$ is $2$-generated). Similarly, two generators of
$\overline{\Gamma}_n$, together with all diagonal matrices
$\mathrm{diag}(\alpha, 1, \dots , 1)$ as $\alpha$ runs over a
generating set for the unit group $\Z_m^*$ of $\Z_m$, generate
$\overline{G}_n$. If $m=2$ or an odd prime power then
$\overline{G}_n$ is $2$-generated. For all $k\geq 3$,
$\GL(n,\Z_{2^k})$ is $4$-generated, and $\GL(n,\Z_{4})$ is
$3$-generated.

The normal closure of $A$ in $B$ is denoted $A^B$. Let $(k,l)$ be
the permutation matrix obtained by swapping rows $k$ and $l$ of
$1_n$.
\begin{lemma}\label{OneTransvectionImpliesAll}
For any $i\neq j$, $E_{n,m}^{\Gamma_n} = \allowbreak \langle
t_{ij}(m)\rangle^{\Gamma_n}$.
\end{lemma}
\begin{proof}
Put $N=\langle t_{ij}(m)\rangle^{\Gamma_n}$. We prove that
$t_{kl}(m) \in N$ for all $k\neq l$. By
Lemma~\ref{tijRelations}~(ii),
\[
t_{kj}(m)= t_{ij}(m)t_{ij}(-m)^{t_{ki}}, \quad k \neq j, i;
\]
so $t_{kj}(m)\in N$. Then $t_{kl}(m) = [t_{kj}(m) , t_{jl}]\in N$
if $k,l\neq j$. Since $t_{kl}(m)=t_{lk}(-m)^{(k,l)d}$ where
$d=\mathrm{diag}(1,\ldots, 1, -1, 1, \ldots, 1)$ with $-1$ in
position $k$, this concludes the proof.
\end{proof}

\begin{proposition}\label{GammaNormalClosureElementaryGroup}
If $n\geq 3$ and $i\neq j$ then $\Gamma_{n,m} = \langle
t_{ij}(m)\rangle^{\Gamma_n} = E_{n,m}^{\Gamma_n}$ (hence
$\Gamma_{n,m} = E_{n,m}^{G_n}$).
\end{proposition}
\begin{proof}
See \cite{BLS}, \cite{BassMilnorSerre}, or \cite{Mennicke}.
\end{proof}
\begin{remark}
For $n,m>1$, $E_{n,m}$ is not normal in $\Gamma_n$.
\end{remark}
\begin{remark}
$E_{n,m_1} \leq E_{n,m_2} \Leftrightarrow \Gamma_{n,m_1} \leq
\Gamma_{n,m_2} \Leftrightarrow$ $m_2 \hspace*{.2mm} \big|
\hspace*{.2mm} m_1$.
\end{remark}

A PCS in $\overline{\Gamma}_{n}$ for $n\geq 3$ is the image under
$\varphi_m$ of a PCS in $\Gamma_n$.
\begin{corollary}\label{PCSForGammaOverZm}
Let $I$ be an ideal of $\Z_m$, so $\Z_m/I \cong \Z_a$ for some
divisor $a$ of $m$. If $n\geq 3$ then the kernel
$\overline{\Gamma}_{n,a}$ of $\varphi_I$ on
$\overline{\Gamma}_{n}=\SL(n,\Z_m)$ is
\[
\{ 1_n + ax \in \overline{\Gamma}_n \st x \in\mathrm{Mat}(n,\Z_m)
\} = \varphi_m(\Gamma_{n,a}) = E_{n,a}^{\overline{\Gamma}_{n}}.
\]
Furthermore, $\overline{\Gamma}_{n,a}= \langle
t_{ij}(a)\rangle^{\overline{\Gamma}_n}=\langle
t_{ij}(a)\rangle^{\overline{G}_n}$ for any $i$ and $j\neq i$.
\end{corollary}

\begin{proposition} \label{VenkaSury}
If $n\geq 3$ then $\Gamma_{n,m}$ has generating set
\begin{equation}\label{VenkaGenSet}
\{ t_{ij}(m)^g \st 1\leq  i< j \leq n, \, g \in \Sigma \}
\end{equation}
where
\[
\Sigma = \{ 1_n, (k,l), 1_n - 2e_{kk} - 2e_{k+1,k+1}+e_{k+1,k} \st
1 \leq k< l \leq n \}.
\]
\end{proposition}
\begin{proof}
See \cite{Venka}.
\end{proof}

We emphasize that the number of generators in \eqref{VenkaGenSet}
does not depend on $m$. The minimal size of a generating set for
$\Gamma_{n,m}$ is unknown. However, by Lemma~\ref{MiNiGensets}
below, this size can be no less than $n^2-1$. As Professor
A.~Lubotzky has pointed out to us, \cite[Theorem~1]{VenkaSharma}
and Lemma~\ref{MiNiGensets} imply that $\Gamma_{n,m}$ has a
generating set of size $n^2+2$. In \cite{Lubotzky86} it is
conjectured that $\Gamma_{n,m}$ for $n\geq 3$ contains a
$2$-generator subgroup of finite index
(cf.~\cite[p.~412]{LongReid}). This conjecture has
been settled affirmatively (see \cite{Meiri}), 
so $\Gamma_{n,m}$ is $(n^2+1)$-generated.

Let $\mathrm{min}(H)$ denote the minimal size of a generating set
of $H$. Although $\mathrm{min}(H)$ can be arbitrarily large
\cite[pp.~355--356]{Venka}, we have
\begin{lemma}\label{Eff2}
Suppose that $n\geq 3$ and $\Gamma_{n,m} \leq H \leq \Gamma_n$.
Then $\mathrm{min}(H)$ is bounded above by a function of $n$, $m$
only.
\end{lemma}
\begin{proof}
This is clear from Proposition~\ref{VenkaSury} and the fact that
$|H : \Gamma_{n,m}|\leq |\SL(n,\Z_m)|$.
\end{proof}

\subsection{Constructing a PCS in an arithmetic subgroup}

Let $n\geq 3$. Our overall strategy rests on knowing some
$\Gamma_{n,m}$ in the arithmetic group $H\leq \Gamma_n$. We show
that such a PCS can always be constructed.
\begin{proposition}\label{Excellent}
$\Gamma_{n,m^2}\le E_{n,m}$, so $|\Gamma_n: E_{n,m}|$ is finite.
\end{proposition}
\begin{proof}
Let $p_{ij}=\abk t_{ij}(m)$ and $s_{ij}=t_{ij}(m^2)$. Then
$\Gamma_{n,m^2}$ is generated by the $s_{ij}$ for $i<\abk j$ and
their conjugates as in Proposition~\ref{VenkaSury}. 
Our goal is to
prove that these all lie in $E_{n,m}$, i.e., that they can be
expressed as words in the $p_{ij}$. Since $s_{ij}^{(k,l)}=
p_{i'j'}^m$ where $i'=i^{(k,l)}$ and $j'=j^{(k,l)}$, it suffices
to look at conjugation by
$c_l=1_n-2e_{ll}-2e_{l+1,l+1}+e_{l+1,l}$ for $l<n$. Furthermore,
if $l,l+1\not\in\{i,j\}$ then $s_{ij}$ and $c_l$ commute: thus it
suffices to consider conjugation of $s_{ij}$ by $c_i$, $c_{i-1}$,
$c_j$, $c_{j-1}$.

First we suppose that the conjugating element has subscript $i$ or
$i-1$. For $j=i+1$ and $a\not \in \abk \{i,i+1\}$,
\begin{equation}
s_{ij}^{c_i} = p_{ai}^{-1}p_{aj} p_{ia}^{-1}p_{ja}^{-1}
p_{aj}^{-1}p_{ai} p_{ja}p_{ia} = [ p_{aj}^{-1}p_{ai},p_{ja}p_{ia}
]. \label{formula6}
\end{equation}
If $j\not=i+1$ we have
\begin{equation}
s_{ij}^{c_i}=(p_{i+1,j}^{-1})^{m-1}p_{i,i+1}p_{i+1,j}^{-1}p_{i,i+1}^{-1}.
\label{formula7}
\end{equation}
For $j\neq i-1$,
\begin{equation}
s_{ij}^{c_{i-1}}=p_{i,i-1}p_{i-1,j}^{-1}p_{i,i-1}^{-1}p_{i-1,j} =[
p_{i,i-1}^{-1},p_{i-1,j} ], \label{formula8}
\end{equation}
while $s_{i,i-1}$ and $c_{i-1}$ commute.

Now suppose that the subscript of the conjugating element is $j$ or
$j-1$. For $j\not=i+1$,
\begin{equation}
s_{ij}^{c_{j-1}}=p_{j-1,j}p_{i,j-1}p_{j-1,j}^{-1}p_{i,j-1}^{m-1}.
\label{formula10}
\end{equation}
If $j=i+1$ then $c_{j-1}=c_i$ and \eqref{formula6} applies.

If $i\not=j+1$ then
\begin{equation}
s_{ij}^{c_j}=p_{j+1,j}^{-1}p_{i,j+1}^{-1}p_{j+1,j}p_{i,j+1}
=\left[ p_{j+1,j},p_{i,j+1} \right], \label{formula11}
\end{equation}
and if $i=j+1$, again as noted above, $s_{ij}=s_{i,i-1}$ and
$c_j=c_{i-1}$ commute.
\end{proof}

The group $\Gamma_n$ has a (finite) presentation $\langle t_{ij} ,
1\leq i, j \leq n, \, i\neq j \st \mathcal{R} \rangle$ where
$\mathcal{R}$ consists of all commutator relations
$[t_{ij},t_{km}] =1$, $[t_{ij},t_{jk}]=t_{ik}$ from
Lemma~\ref{tijRelations} (ii) and (iii), with a single extra
relation $(t_{12}t_{21}^{-1}t_{12})^4=1$ \cite[Corollary
10.3]{Milnor}.
\begin{lemma}\label{CanFindIndexOfFIHinGamma}
Given $H\fis \Gamma_n$ we can find an elementary group in $H$.
\end{lemma}
\begin{proof}
Express each generator of $H$ as a product of transvections (for
which see, e.g., \cite[p.~99]{Humphreys}). Then the Todd-Coxeter
procedure with input $\Gamma_n$ and $H$ terminates, returning
$m=|\Gamma_n:H|$. So for all $i$, $j$ and known $l$ we have
$t_{ij}(l)= \abk t_{ij}(1)^{l}\in \abk H$ 
($l=\abk \mathrm{lcm}\{1, \ldots , m \}$ say). 
Hence $E_{n,l}\leq H$.
\end{proof}

Using Proposition~\ref{Excellent}, we rescue one item (slightly
generalized) from the proof of
Lemma~\ref{CanFindIndexOfFIHinGamma}.
\begin{lemma}\label{CfLS}
If $\, |\Gamma_n:H|\leq m$ then $\Gamma_{n,l^2}\leq H$ where $l
=\mathrm{lcm} \{ 1, \ldots , m\}$.
\end{lemma}

Proposition~\ref{Excellent} and
Lemma~\ref{CanFindIndexOfFIHinGamma} yield the promised
\begin{corollary}\label{FindingPCSIsDecidable}
Construction of a PCS in $H\fis \Gamma_n$ is decidable.
\end{corollary}

\subsection{Maximal congruence subgroups}

In this subsection $n\geq 3$ and $G_n= \abk \glnz$.
\begin{lemma}\label{GCD}
Let $m_1$, $m_2$ be positive integers, $m= \mathrm{gcd}(m_1,m_2)$,
and $l = \mathrm{lcm}(m_1,m_2)$. Then
\begin{itemize}
\item[{\rm (i)}]  $\Gamma_{n,m_1}\Gamma_{n,m_2}= \Gamma_{n,m}$.
\item[{\rm (ii)}] $\Gamma_{n,m_1}\cap
\Gamma_{n,m_2}=\Gamma_{n,l}$.
\end{itemize}
\end{lemma}
\begin{proof}
(i) \hspace*{-.15mm} For $x\in \Gamma_n$ and integers $a$, $b$
such that $am_1+bm_2 = m$,
\[
t_{ij}(m)^x = (t_{ij}(m_1)^x)^a\cdot (t_{ij}(m_2)^x)^b \in
\Gamma_{n,m_1}\Gamma_{n,m_2}.
\]
Thus $\Gamma_{n,m}=\Gamma_{n,m_1}\Gamma_{n,m_2}$ by
Proposition~\ref{GammaNormalClosureElementaryGroup}.

(ii) \hspace*{-.15mm} Certainly $\Gamma_{n,l} \leq
\Gamma_{n,m_1}\cap \Gamma_{n,m_2}$. The reverse containment is
just the Chinese Remainder Theorem.
\end{proof}

\begin{corollary}\label{MaxPrincipalCS}
If $H\fis G_n$ then $H$ contains a unique maximal PCS (of $\,
\Gamma_n$): there is a positive integer $m$ such that
$\Gamma_{n,m}\leq H$, and $\Gamma_{n,k} \leq \abk H \Rightarrow
\Gamma_{n,k} \leq \abk \Gamma_{n,m}$.
\end{corollary}
\begin{remark}
If $H$ has maximal PCS $\Gamma_{n,m}$ and $\mathrm{gcd}(k,m)=1$
then $\varphi_k(H) = \SL(n,\Z_k)$. Hence we know $\nu$ such that
$\varphi_p(H) = \SL(n,p)$ for all primes $p>\nu$; cf.~the query
raised at the foot of \cite[p.~126]{Lubotzky97}.
\end{remark}
\begin{remark}
Although $H$ similarly contains a unique maximal elementary
subgroup $E_{n,m}$, the $\Gamma_n$-normal closure of $E_{n,m}$
need not be the maximal PCS in $H$, nor even be in $H$.
\end{remark}
\begin{remark}
Lemma~\ref{CfLS} provides an upper bound on $m$ such that
$\Gamma_{n,m}$ is the maximal PCS of an arithmetic group in
$\Gamma_n$; cf.~\cite[Proposition~6.1.1, p.~115]{LubotzkySegal}.
\end{remark}

\begin{lemma}\label{MaxPCSEasy}
Each subgroup of $\, \overline{G}_n= \GL(n,\Z_m)$ contains a
(perhaps trivial) unique maximal PCS of $\,
\overline{\Gamma}_n=\SL(n,\Z_m)$. In more detail, suppose that $\,
\Gamma_{n,m}\leq H\leq \Gamma_n$ and $\Gamma_{n,r}$ is the maximal
PCS in $H$; then $\overline{\Gamma}_{n,r} =\abk
\varphi_m(\Gamma_{n,r})$ is the maximal PCS in
$\overline{H}=\varphi_m(H)$.
\end{lemma}
\begin{proof}
Since $\Gamma_{n,m}\leq \Gamma_{n,r}$, we have that $r$ divides
$m$, and so $\overline{\Gamma}_{n,r}$ is a PCS in $\overline{H}$.
Corollary~\ref{PCSForGammaOverZm} tells us that each PCS in
$\overline{H}$ has the form $\overline{\Gamma}_{n,k} =
\varphi_m(\Gamma_{n,k})$ for some $k \hspace*{.2mm} \big|
\hspace*{.2mm} m$. Moreover $\Gamma_{n,k}\leq H$, because $H$
contains $\ker \hspace*{.1mm} \varphi_m$. Hence
$\overline{\Gamma}_{n,r}$ is as claimed.
\end{proof}

\subsection{Subnormal structure}

Let $Z_{n,I}$ denote the full preimage of the center (scalar
subgroup) of $\GL(n,R/I)$ in $G_n=\GL(n,R)$ under $\varphi_I$. As
per \cite[p.~166]{JW}, the \emph{level} $\ell(h)$ of
$h=\allowbreak (h_{ij}) \in \allowbreak G_n$ is the ideal of $R$
generated by
\[
\{ h_{ij} \st i\neq j, \, 1\leq i , j \leq n\} \cup \{
h_{ii}-h_{jj} \st  1\leq i , j \leq n \}.
\]
Then $\ell(A) :=\allowbreak \sum_{a\in A} \ell(a)$ for $A\subseteq
\allowbreak G_n$. So $\ell(A)$ is the smallest ideal $I$ such that
$A\subseteq \abk Z_{n,I}$. When $R$ is a principal ideal ring we
 write $b$ in place of $I=bR$. For $R=\Z$ or $\Z_m$, $\ell(A)$
may be defined unambiguously as the non-negative integer or
integer modulo $m$ that generates $\ell(A)$; e.g., $\ell(Z_{n,k})
= \abk \ell(\Gamma_{n,k}) = \abk k$.
\begin{lemma}\label{LevelViaGenerators}
If $H=\gpess\leq G_n$ then $\ell(H) = \ell(S)$.
\end{lemma}
\begin{proof}
It is evident from the definitions that $\ell(S) \subseteq
\ell(H)$ and $\ell(ab)\subseteq \ell(a) + \ell (b)$ for $a$,
$b\in\abk G_n$. Since $\ell(a) = \ell(a^{-1})$ by
\cite[Lemma~1]{JW}, $\ell(H) \subseteq \ell(S)$ as required.
\end{proof}

From now on in this subsection, $n\geq 3$ and $R =\Z$ or $\Z_m$.
We write $H \, \mathrm{sn}\, G$ to denote that $H\leq \abk G$ is
subnormal. The \emph{defect} of $H$ is the least $d$ such that
there exists a series 
$H= H_0 \unlhd \allowbreak H_1 \unlhd \abk
\cdots \abk \unlhd H_{d-1} \unlhd\abk H_{d} = G$.
\begin{theorem}\label{FurtherSubnormalFacts}
 $H \, \mathrm{sn}\, G_n$ if and only if \
\begin{equation}\label{SubnormalChain}
\Gamma_{n,k^e} \leq H \leq Z_{n,k}
\end{equation}
for some $k$, $e$. If \eqref{SubnormalChain} holds then $d\leq
e+1$ where $d$ is the defect of $H$, and the least possible $e$ is
bounded above by a function of $n$ and $d$ only.
\end{theorem}
\begin{proof}
See \cite[Corollary 3]{JW}.
\end{proof}

Although non-scalar subnormal subgroups of $\glnz$ have finite
index, this is not true for $n=\abk 2$; the normal closure of
$E_{2,m}$ in $\SL(2,\Z)$ has infinite index for $m>5$~\cite[p.~31]{Mennicke}.
\begin{theorem}\label{RelationToMaximalPCS}
Let $H$ be a subgroup of $G_n$ of level $l\geq 1$, with maximal
PCS $\Gamma_{n,r}$. Then $H \, \mathrm{sn}\, \abk G_n$ if and only
if $r \hspace*{.2mm} \big| \hspace*{.2mm} l^e$ for some $e$. In
that event, the defect of $H$ is bounded above by $e'+1$ where
$e'$ is the least such $e$.
\end{theorem}
\begin{proof}
If $H$ is subnormal then $lR \subseteq kR$ and $\Gamma_{n,k^e}\leq
\Gamma_{n,r}$ for $k$, $e$ as in
Theorem~\ref{FurtherSubnormalFacts}; so $k \hspace*{.2mm} \big|
\hspace*{.2mm} l$ and $r\hspace*{.2mm} \big| \hspace*{.2mm} k^e$.
Conversely, if $r\hspace*{.2mm} \big| \hspace*{.2mm} l^e$ then $H$
satisfies \eqref{SubnormalChain} with $k=l$.
\end{proof}

\begin{lemma}[\cite{JW}, p.~165]
\label{NilpotentQuotient} $Z_{n,l}/\Gamma_{n,l^e}$ is nilpotent of
class at most $e$.
\end{lemma}

We now consider normality.
\begin{lemma}\label{IfNormalThenSameLevels}
If $\,\Gamma_{n,l} \leq H \leq Z_{n,l}$ then $H\unlhd G_n$ and $l
= \ell(H) = $ the level of the maximal PCS in $H$.
\end{lemma}
\begin{proof}
We first observe that $l = \ell(\Gamma_{n,l}) \geq \ell(H) \geq
\ell(Z_{n,l}) = l$. Let $\Gamma_{n,r}$ be the maximal PCS in $H$.
Then $r \hspace*{.2mm} \big| \hspace*{.2mm} l$; and $l
\hspace*{.2mm} \big| \hspace*{.2mm} r$ because $\Gamma_{n,r}\leq
Z_{n,l}$.
\end{proof}

\begin{lemma}\label{Sandwich}
Suppose that $H\leq G_n$ has level $l$. Then
\begin{itemize}
\item[{\rm (i)}] $\Gamma_{n,l}\leq H^{G_n} \leq Z_{n,l}$. \item[]
\vspace{-12.5pt} \item[{\rm (ii)}] $H^{G_n} = \langle H,
\Gamma_{n,l}\rangle$.
\end{itemize}
\end{lemma}
\begin{proof}
(i)  \hspace*{.01mm} The inclusion $H^{G_n} \leq Z_{n,l}$ is
clear. If $h\in H$ has level $a$ then $t_{12}(a)\in \langle
h\rangle^{G_n}$ by Theorems~1 and 4 of \cite{BrennerIII}. As a
consequence, $t_{12}(l)\in H^{G_n}$. Now this part is assured by
Proposition~\ref{GammaNormalClosureElementaryGroup} and
Corollary~\ref{PCSForGammaOverZm}.

(ii) \hspace*{.2mm} Let $L= \langle H, \Gamma_{n,l}\rangle$. Since
$L\unlhd G_n$ (Lemma~\ref{IfNormalThenSameLevels}), $H^{G_n}\leq
\abk L$. Also $L\leq H^{G_n}$ by (i).
\end{proof}
\begin{corollary}\label{CorNormalityCriterion}
$H\unlhd G_n$ if and only if $\ell(H)$ is the level of the maximal
PCS in $H$.
\end{corollary}
\begin{proposition}\label{NormalinSLiffNormalinGL}
Lemma{\em ~\ref{Sandwich}} remains true with $G_n$ replaced by
$\Gamma_n=\SL(n,R)$. That is, $H^{\Gamma_n} = \abk H^{G_n}$, and
so $H\leq \Gamma_n$ is normal in $\Gamma_n$ precisely when it is
normal in $G_n$.
\end{proposition}

\section{Matrix groups over $\Z_m$}
\label{AlgorithmsforgroupsoverZm}

\subsection{Relevant theoretical results}
\label{VarStructureZm}

Let $m= p_1^{k_1} \cdots p_t^{k_t}$ where the $p_i$ are distinct
primes and $k_i\geq \abk 1$. We define a ring isomorphism $\chi:
\Z_m \rightarrow \Z_{p_1^{k_1}} \oplus \cdots \oplus
\Z_{p_t^{k_t}}$ by $\chi(a) = (a_1, \ldots , a_t)$ where
$0\leq\abk a \leq \abk m-1$, $0\leq \abk a_i \leq \abk
p_i^{k_i}-1$, and $a_i\equiv \abk a \mod p_i^{k_i}$.
\begin{lemma} \label{FundamIsom}\
\begin{itemize}
\item[{\rm (i)}] $\chi$ extends to an isomorphism of $\,
\mathrm{Mat}(n,\Z_m)$ onto $\oplus_{i=1}^t
\mathrm{Mat}\big(n,\Z_{p_i^{k_i}}\big)$, which restricts to
isomorphisms $\mathrm{GL}(n,\Z_m) \rightarrow \times_{i=1}^t
\mathrm{GL}\big(n,\Z_{p_i^{k_i}}\big)$ and $\mathrm{SL}(n,\Z_m)
\rightarrow \times_{i=1}^t \mathrm{SL}\big(n,\Z_{p_i^{k_i}}\big)$.
\item[{\rm (ii)}] Let $I = \langle a \rangle$ be an ideal of
$\Z_m$, and let $I_i$ be the ideal of $\Z_{p_i^{k_i}}$ generated
by $a_i \equiv\abk a \mod \abk p_i^{k_i}$. Denote by $K_I$,
$K_{I_i}$ the kernels of $\varphi_I$, $\varphi_{I_i}$ on
$\GL(n,\Z_m)$, $\mathrm{GL}\big(n,\Z_{p_i^{k_i}}\big)$
respectively. Then
\[
\chi (K_I) = \times_{i=1}^t K_{I_i} \quad \text{and} \quad \chi
(K_I\cap \SL(n,\Z_m)) = \times_{i=1}^t (K_{I_i}\cap \SL(n,\Z_m)).
\]
\end{itemize}
\end{lemma}

For $i\geq 1$,
\[
M_{p,i} = \{ h \in \GL(n,\Z_{p^k}) \st h \equiv 1_n \mod p^i\},
\qquad N_{p,i} = \mathrm{SL}(n,\Z_{p^k})\cap M_{p,i}
\]
are normal subgroups of $\GL(n,\Z_{p^k})$.
\begin{lemma}[\mbox{Cf.} Corollary~\ref{PCSForGammaOverZm}]
\label{OrderCalcs} If $I$ is the ideal of $\Z_{p^k}$ generated by
$p^{i}$, then $\varphi_I : \abk \GL(n,\Z_{p^k}) \rightarrow \abk
\GL(n, \Z_{p^i})$ and $\varphi_I: \SL(n,\Z_{p^k})\rightarrow \abk
\SL(n, \Z_{p^i})$ are surjective, with kernels $M_{p,i}$,
$N_{p,i}$ respectively.
\end{lemma}

The notation $M_{p,i}$, $N_{p,i}$ supersedes our earlier notation for
principal congruence subgroups in this special case. Let $d_j(a)=
1_n + ae_{jj}\in \mathrm{Mat}(n,\Z_m)$.
\begin{lemma}\label{ModifiedOC2}
Suppose that $i<j\leq 2i$ and $j\leq k$. Then $M_{p,i}/M_{p,j}
\cong C_{p^{j-i}}^{n^2}$, and $N_{p,i}/N_{p,j}$ has a subgroup
isomorphic to $C_{p^{j-i}}^{n^2-1}$.
\end{lemma}
\begin{proof}
Treating $\mathrm{Mat}(n,\Z_{p^{j-i}})$ as an additive group, we
confirm that $\theta_j : M_{p,i} \rightarrow
\mathrm{Mat}(n,\Z_{p^{j-i}})$ defined by $\theta_j (1_n +p^i x) =
\varphi_{p^{j-i}} (x)$ is a homomorphism with kernel $M_{p,j}$.
Now $t_{rs}(p^i) \in \abk N_{p,i}$ and $d_r(p^i) \in\abk M_{p,i}$,
so $\theta_j$ is surjective. Since $N_{p,i}$ contains $1_n +\abk
p^i(e_{rr}-\abk e_{r+1,r+1}+\abk e_{r,r+1}-e_{r+1,r})$, the second
assertion follows too.
\end{proof}

\begin{lemma}\label{NiftyGHK}
$[M_{p,i},M_{p,j}] = [N_{p,i},N_{p,j}] = N_{p,i+j}$.
\end{lemma}
\begin{proof}
(\mbox{Cf.} Lemma~\ref{NilpotentQuotient}.) Let $a=1_n +p^ix \in
M_{p,i}$ and $b=1_n +p^jy \in M_{p,j}$. For some $z$, and
$\bar{x}$, $\bar{y}$ such that $a^{-1} =\abk 1_n +p^i\bar{x}$ and
$b^{-1} = 1_n +p^j\bar{y}$, we have
\begin{align*}
[a,b] & = (1_n+p^i\bar{x}+p^j
\bar{y}+p^{i+j}\bar{x}\bar{y})(1_n+p^ix+p^j y+p^{i+j}xy)\\
\abk & = 1_n+p^i(x+\bar{x})+p^{2i}\bar{x}x + p^j(y+\bar{y}) +
p^{2j}\bar{y}y + p^{i+j}z\\ \abk & =1_n+p^{i+j}z.
\end{align*}
Therefore $[M_{p,i},M_{p,j}]\leq M_{p,i+j} \cap \abk
\SL(n,\Z_{p^k}) = N_{p,i+j}$. Also $t_{21}(p^{i+j}) =
[t_{23}(p^i),t_{31}(p^j)] \in \abk [N_{p,i},N_{p,j}]\unlhd \abk
\SL(n,\Z_{p^k})$; thus $N_{p,i+j} \leq \abk [N_{p,i},N_{p,j}]$ by
Corollary~\ref{PCSForGammaOverZm}.
\end{proof}

\begin{lemma}\label{MiOrders}\
\begin{itemize}
\item[{\rm (i)}] $|M_{p,i}|= p^{n^2(k-i)}$. \item[{\rm (ii)}]
$|\GL(n,\Z_{p^k})| = |\GL(n,p)|\cdot p^{n^2(k-1)}$.
\end{itemize}
\end{lemma}
\begin{proof}
Lemma~\ref{ModifiedOC2} takes care of (i). By
Lemma~\ref{OrderCalcs}, we then get (ii).
\end{proof}
\begin{corollary}
If $\, 2i>k$ then $M_{p,i}$ is abelian of exponent $p^{k-i}$.
\end{corollary}

The next two corollaries use Lemma~\ref{FundamIsom}. Let $a =
p_1^{j_1}\cdots p_t^{j_t}$ where $0\leq j_i\leq k_i$. Note that
$a_i \equiv\abk a \mod \abk p_i^{k_i}$ generates the ideal
$\langle p_i^{j_i}\rangle$ of $\Z_{p_i^{k_i}}$. Set $M_{p_i,0} =
\GL(n,\Z_{p_i^{k_i}})$ and $N_{p_i,0} =\abk
\SL(n,\Z_{p_i^{k_i}})$.
\begin{corollary}\label{PCSGOrder}\
\begin{itemize}
\item[{\rm (i)}] $|\GL(n,\Z_{m})| =
\prod_{i=1}^t\big(|\GL(n,p_i)|\cdot p_i^{n^2(k_i-1)}\big)$.

\item[] \vspace{-12.5pt}

\item[{\rm (ii)}] The PCS of $\GL(n,\Z_{m})$ of level $a$ has
order $\prod_{i=1}^t |M_{p_i,j_i}|$.
\end{itemize}
\end{corollary}

\begin{lemma}\label{NiOrders} \
\begin{itemize}
\item[{\rm (i)}] $|\SL(n,\Z_{p^k})| = |\SL(n,p)|\cdot
p^{(n^2-1)(k-1)}$. \item[{\rm (ii)}] For $i\geq 1$,
$N_{p,i}/N_{p,i+1}\cong C_{p}^{n^2-1}$ and $|N_{p,i}|=
p^{(n^2-1)(k-i)}$.
\end{itemize}
\end{lemma}
\begin{proof}
The unit group of $\Z_{p^k}$ has order $(p-1)p^{k-1}$, so 
Lemma~\ref{MiOrders} (ii) gives (i).
By Lemma~\ref{ModifiedOC2}, $|N_{p,i}/N_{p,i+1}|\geq p^{n^2-1}$. 
Thus, if
$|N_{p,j}/N_{p,j+1}|\neq \abk p^{n^2-1}$ for some $j$ then
$|N_{p,1}|> \abk p^{(n^2-1)(k-1)}$, which contradicts (i) by
Lemma~\ref{OrderCalcs}.
\end{proof}

\begin{corollary}\label{PCSSLOrder}\
\begin{itemize}
\item[{\rm (i)}] $|\SL(n,\Z_{m})| =
\prod_{i=1}^t\big(|\SL(n,p_i)|\cdot p_i^{(n^2-1)(k_i-1)}\big)$.

\item[] \vspace{-12.5pt}

\item[{\rm (ii)}] The PCS of $\SL(n,\Z_{m})$ of level $a$ has
order $\prod_{i=1}^t |N_{p_i,j_i}|$.
\end{itemize}
\end{corollary}

Define subsets
\[
S_{c} \! =\{ t_{rs}(c) , \hspace*{1pt} 1_n+c\hspace*{.5pt}
(e_{uu}+e_{u,u+1}-e_{u+1,u}-e_{u+1,u+1})  \st r\neq s, \, 1\leq r,
s \leq n,\, 1\leq u \leq n-1 \}
\]
\noindent of $\mathrm{SL}(n,\Z_m)$ and
\begin{align*}
T_{c}  =\{  t_{rs}(c) , d_1(c), \ldots , d_n(c) \st r\neq s, \,
1\leq r, s \leq n \}
\end{align*}
of $\mathrm{Mat}(n,\Z_m)$. We see that $T_c \leq \GL(n,\Z_m)$ if
and only if $1+c$ is a unit of $\Z_m$.
\begin{lemma}\label{MiNiGensets}
Suppose that $1\leq i<k$.
\begin{itemize}
\item[{\rm (i)}] $N_{p,i}$ has minimal size generating set $S_{p^i}$,
so $\mathrm{min}(N_{p,i}) = n^2 - 1$. \item[{\rm (ii)}] Unless
$p=2$, $k\geq 3$, and $i = 1$, $\mathrm{min}(M_{p,i}) = n^2$ and
$M_{p,i}$ has minimal size generating set $T_{p^i}$. \item[{\rm (iii)}]
$M_{2,1}$ for $k\geq 3$ has minimal size generating set $T_{2} \cup \{
\mathrm{diag}(-1,1,\ldots , 1)\}$ of size $n^2+1$.
\end{itemize}
\end{lemma}
\begin{proof}
In the proof of Lemma~\ref{ModifiedOC2} we saw that $N_{p,i}=
\langle S_{p^i}, N_{p,2i}\rangle$. Since $N_{p,i}$ is nilpotent
with derived group $N_{p,2i}$ by Lemma~\ref{NiftyGHK}, we have
$N_{p,i}= \langle S_{p^i}\rangle$. So
$\mathrm{min}(N_{p,i})=\mathrm{min}(N_{p,i}/N_{p,i+1})=\abk n^2-1$
by Lemma~\ref{NiOrders}~(ii).

The rest of the proof is along similar lines. Note that $M_{p,i}=
\langle T_{p^i}, M_{p,2i}\rangle$, and $M_{p,2i}/N_{p,2i}$ is
trivial when $2i\geq k$, or cyclic of order $p^{k-2i}$ generated
by the coset of $d_1(p^{2i})$ otherwise. Also $1+p^{2i} \in
\langle 1+p^i\rangle\leq \abk \Z_{p^k}^*$ unless $p=2$, $k\geq 3$,
and $i=1$; whereas $5 \in \langle -1, 3\rangle = \Z_{2^k}^*$ for
$k\geq 3$. Therefore $M_{p,i}=\abk \langle T_{p^i},
N_{p,2i}\rangle = \langle T_{p^i}\rangle$ in (ii). Since
$|T_{p^i}|=n^2$ and $M_{p,i}/M_{p,i+1}$ has rank $n^2$, this
proves (ii). The verification of (iii) is left as an exercise.
\end{proof}

\begin{proposition}\label{AltMiNiGensets}
Let $H$, $K$ be non-trivial principal congruence subgroups of
level $a = p_1^{j_1}\cdots p_t^{j_t}\neq m$ in $\GL(n,\Z_m)$,
$\SL(n,\Z_m)$ respectively, where $1\leq j_i\leq k_i$ for 
all $i$. 
Then
\begin{itemize}
\item[{\rm (i)}] $\mathrm{min}(H) = n^2$; unless $k_2\geq 3$ and
the Sylow $2$-subgroup of $\chi(H)$ is $M_{2,1}$, in which
case $\mathrm{min}(H) = n^2+1$. 
\item[{\rm (ii)}] $\mathrm{min}(K)= n^2-1$.
\end{itemize}
\end{proposition}
\begin{proof}
If $X$, $Y$ are groups of coprime order with generating
sets $\{x_1, \ldots, x_{r_1}\}\subseteq \abk X$ and 
$\{ y_1,\ldots, y_{r_2}\}\subseteq Y$ of minimal size, 
where $r_1\leq r_2$, then $\mathrm{min}(X\times Y) = r_2$. 
Indeed
\[
X\times Y = \langle (x_i,y_i), (1, y_j) :  
1\leq i \leq r_1 ;\, r_1+1\leq j \leq r_2\rangle.
\]
The result follows from Lemmas~\ref{FundamIsom}~(ii) and
\ref{MiNiGensets}. 
\end{proof}
\begin{remark}\
\begin{itemize}
\item[(i)] If $j_i = 0$ for any $i$ then 
a full GL ($4$-generated) or SL ($2$-generated)
appears as a factor in $H$ or $K$.
\item[(ii)] The proof of Proposition~\ref{AltMiNiGensets} shows how to
construct minimal size generating sets for $H$ and $K$ with the aid of
Lemma~\ref{MiNiGensets}. Note that we get a generating set for a
PCS in $\SL(n,\Z_m)$ by reducing \eqref{VenkaGenSet} in
Proposition~\ref{VenkaSury} modulo $p$.
\end{itemize}
\end{remark}

\subsection{Computing in $\GL(n,\Z_m)$}\label{DataStructureZm}

As above, suppose that $m\geq 2$ has prime factorization
$\prod_{i=1}^tp_i^{k_i}$. Let $\chi$ be the isomorphism introduced
just before Lemma~\ref{FundamIsom}. We identify $H\leq
\GL(n,\Z_m)$ with $\chi(H)$.

To compute with $H$, we use composition tree methods and the data
structure from~\cite{Hulpke2013}. The latter consists of an
effective homomorphism into $\times_{i=1}^t\GL(n,p_i)$ whose
kernel $K$ is the solvable radical of $H$, and a polycyclic
generating sequence (PCGS) for $K$. Data structures for the images
of the projections of $H$ modulo $p_i^{k_i}$ can be combined into
a data structure for $H$. We therefore assume that $m=p^k$.

Clearly $H/K$ is isomorphic to a quotient of $\varphi_p(H)\leq\abk
\glnp$, and a PCGS for $\varphi_p(H)$ gives the initial terms of a
PCGS for $K$; the rest are found by reductions modulo $p^e$
(cf.~Subsection~\ref{VarStructureZm}). As we have seen, if $M$ is
the kernel of reduction modulo $p^e$ and $N$ the kernel of
reduction modulo $p^{e+1}$, then $M/N$ is described by matrices
$1_n+p^e x$ for $x\in\mathrm{Mat}(n,p)$, which multiply by
addition of their $x$-parts. A PCGS for the elementary abelian
group $M/N$ can be determined easily by linear algebra.

\subsection{Subnormal structure}
\label{SubNStructureZm}

Let $n\geq 3$. We adhere to previous notation and conventions.

Let ${\tt Level}$ be a function that returns $\ell(H)$ for a
subgroup $H=\abk\gpess$ of $G_n = \abk \GL(n,\Z_m)$; see
Lemma~\ref{LevelViaGenerators}.

\vspace*{10pt}

${\tt MaxPCS}(H)$

\vspace*{1mm}

Input: $H\leq G_n$.

Output: a generating set for a maximal PCS of $\Gamma_n=
\SL(n,\Z_m)$ in $H$.

\vspace*{.25mm}

\begin{enumerate}

\item $l:= {\tt Level} (H)$.

\abk

\item  \label{MembTestSteppe}  If $l = 0$ then return $1_n$,

else return a generating set $L$ for the PCS of level $a$ in
$\Gamma_n$ as given by Proposition~\ref{AltMiNiGensets},

where $a$ is minimal subject to $a$ dividing $m$, $l$ dividing
$a$, and $L\subseteq H$.
\end{enumerate}

\vspace*{10pt}

Step (2) requires membership testing. As an application of ${\tt
MaxPCS}$, we have

\vspace*{10pt}

${\tt IsSpecialLinear}(H)$

\vspace*{1mm}

Input: $H\leq \Gamma_n$.

Output: ${\tt true}$ if $H=\Gamma_n$; ${\tt false}$ otherwise.

\vspace*{.25mm}

\begin{enumerate}
\item[] If ${\tt Level}({\tt MaxPCS}(H)) = 1$ then return ${\tt
true}$
\begin{itemize}
\item[] else return ${\tt false}$.
\end{itemize}

\end{enumerate}

\vspace*{10pt}

The following reiterates Theorem~\ref{RelationToMaximalPCS}.

\vspace*{10pt}

${\tt IsSubnormal}(H)$

\vspace*{1mm}

Input: $H\leq G_n$.

Output: ${\tt true}$ and an upper bound $d$ on the defect of $H$
if $H \, \mathrm{sn}\, G_n$; ${\tt false}$ otherwise.

\vspace*{.25mm}

\begin{enumerate}
\item $l_1:= {\tt Level}(H)$, $l_2:={\tt Level}({\tt MaxPCS}(H))$.

\abk

\item If $\nexists$ $ e$ such that $l_2 \hspace*{.2mm} \big|
\hspace*{.2mm} l_1^e$ then return ${\tt false}$,

else return ${\tt true}$ and $d := e'+1$ where $e':=$ the least
$e$ such that $l_2 \hspace*{.2mm} \big| \hspace*{.2mm} l_1^e$.
\end{enumerate}

\vspace*{7.5pt}

\begin{remark}
Let $H\leq \Gamma_n$. Obviously $H \, \mathrm{sn}\, \Gamma_n$ if
and only if $H \, \mathrm{sn}\, G_n$. The defect of $H$ as a
subnormal subgroup of $\Gamma_n$ is either equal to or one less
than its defect as a subgroup of $G_n$.
\end{remark}

${\tt NormalClosure}(H)$ returns the normal closure of $H$ in
$G_n$ according to Lemma~\ref{Sandwich}. ${\tt IsNormal}$ tests
whether $H \unlhd G_n$, returning ${\tt true}$ if and only if
$l_2=l_1$ (Corollary~\ref{CorNormalityCriterion}).

By Proposition~\ref{NormalinSLiffNormalinGL}, ${\tt
NormalClosure}$ also returns the normal closure in $\Gamma_n$ of
$H\leq \Gamma_n$, and ${\tt IsNormal}$ tests whether $H
\unlhd\Gamma_n$.

We can list the subnormal subgroups of $G_n$ in $H$.

\vspace*{10pt}

${\tt NormalSubgroups}(H,l)$

\vspace*{1mm}

Input: $H\leq G_n$ and a positive integer $l$.

Output: all normal subgroups of $G_n$ in $H$ of level $l$.

\vspace*{.25mm}

\begin{enumerate}
\item $r := {\tt Level}({\tt MaxPCS}(H))$.

\abk

\item If  $r$ does not divide $l$ then return $\emptyset$.

\abk

\item $\mathcal{L}:=$ a list of all subgroups of $\varphi_l(H)
\cap \varphi_l(Z_{n,l})$.

\abk

\item Return the full preimage of $\mathcal{L}$ in $H$ under
$\varphi_l$.
\end{enumerate}

\vspace*{10pt}

We sketch a more general method. Let $\mathcal{L}_{a,b}$ be
the list of all $K$ such that $\Gamma_{n,b}\leq \abk K\leq \abk
H\cap \abk Z_{n,a}$. Define $\mathcal{L} = \abk
\bigcup_k\mathcal{L}_{k,k^{t}}$ where $k$ ranges over the
multiples of $\ell(H)$ dividing $m$, and $t=t(k)$ is maximal
subject to $r \hspace*{.2mm} \big| \hspace*{.2mm} k^{t}$. Then
$\mathcal L$ is a complete list of subnormal subgroups of $G_n$ in
$H$. By Lemma~\ref{NilpotentQuotient}, $\mathcal{L}_{k,k^t}$
consists of preimages of subgroups of the nilpotent group
$\varphi_{k^t}(Z_{n,k})$. Redundancies in $\mathcal{L}$ are
removed using $\mathcal{L}_{k_1,k_1^{t_1}}\cap\abk
\mathcal{L}_{k_2,k_2^{t_2}}=
\mathcal{L}_{\mathrm{lcm}(k_1,k_2),\mathrm{gcd}(k_1,k_2)^t}$ where
$t  = \mathrm{min}(t_1,t_2)$, by Lemma~\ref{GCD}.

\section{Computing with arithmetic groups in $\slnz$}
\label{AlgorithmsArithmetica}

\subsection{Decidability}

An arithmetic subgroup $H$ of an algebraic $\Q$-group $G\leq
\GL(n,\mathbb{C})$ is `explicitly given' if  (i)~an upper bound on
$|G_\Z : H|$ is known, and  (ii)~membership testing in $H$ is
possible; i.e., for any $g\in G_\Z$ it can be decided whether
$g\in H$~\cite[pp.~531--532]{GSI}. Conditions (i) and (ii) were
assumed in \cite{GSI} to prove decidability of algorithmic
problems for $H$. As the next lemma shows, these conditions are
equivalent to knowing a PCS in $H$. Such a PCS can always be
found: see Corollary~\ref{FindingPCSIsDecidable}.
\begin{lemma}\label{GSExpGEquiv}
Let $H\fis \Gamma_n$. The following are equivalent.
\begin{itemize}
\item[{\rm (i)}] A positive integer $m$ such that
$\Gamma_{n,m}\leq H$ is known. \item[{\rm (ii)}] An upper bound on
$|\Gamma_n:H|$ is known, and testing membership of $x\in \Gamma_n$
in $H$ is decidable.
\end{itemize}
\end{lemma}
\begin{proof}
(i) $\Rightarrow$ (ii). $|\Gamma_n:H|\leq |\SL(n,\Z_m)|$, and
$x\in H$ if and only if $\varphi_m(x)\in \varphi_m(H)$.

(ii) $\Rightarrow$ (i). Suppose that $|\Gamma_n:H|\leq r$. For
$g\in\abk \Sigma$ as in Proposition~\ref{VenkaSury} and each pair
$i,j$, after no more than $r$ rounds we are guaranteed to find
positive integers $r_{g,i,j}\leq r$ such that $t_{ij}(r_{g,i,j})^g
= \abk (t_{ij}^g)^{r_{g,i,j}} \in \abk H$. Thus, if $m$ is any
common multiple of the $r_{g,i,j}$ then $\Gamma_{n,m} \leq H$.
\end{proof}

\begin{proposition}\label{WhoKnew}
If $H$ is a finite index subgroup of $\, \Gamma_n$ specified by a
finite generating set then testing membership of any $g\in
\Gamma_n$ in $H$ is decidable.
\end{proposition}
\begin{proof}
This follows from Corollary~\ref{FindingPCSIsDecidable} and
Lemma~\ref{GSExpGEquiv}.
\end{proof}

Of course, a key problem is (AT), arithmeticity testing: if $H$ is a
finitely generated subgroup of $\Gamma_n$, determine whether
$|\Gamma_n:H|$ is finite.
We are unaware of any proof that (AT) is
decidable---although it seems not to be \cite{Miller}.
Nonetheless, (AT) is decidable when $G$ is solvable
\cite{DeGrDetF}.

\subsection{Algorithms for arithmetic groups}
\label{ZProcedures}

Now we design algorithms for $H\leq \Gamma_n = \abk
\slnz$, $n\geq 3$, given by a finite generating set.

By Corollary~\ref{FindingPCSIsDecidable} (and the proof of
Lemma~\ref{CanFindIndexOfFIHinGamma}), we have a procedure ${\tt
LevelPCS}(H)$ that returns the level of a PCS in 
$H$. It depends
on representing elements of $\Gamma_n$ as products of
transvections. Once we know ${\tt LevelPCS}(H)=m$, say, then ${\tt
GeneratorsPCS}(m)$ returns a generating set for $\Gamma_{n,m}$ as
in Proposition~\ref{VenkaSury}.

Let $\overline{H}=\abk \varphi_m(H) \leq\abk \overline{\Gamma}_n =
\SL(n,\Z_m)$. Lemma~\ref{MaxPCSEasy} underpins the following,
which finds the maximal PCS $\Gamma_{n,r}$ in $H$. (To improve
efficiency we could substitute $r$ for $m$ in algorithms of this
section.)

\vspace*{10pt}

${\tt MaxPCS}(H,m)$

\vspace*{1mm}

Input: $H\leq \Gamma_n$ such that $\Gamma_{n,m}\leq H$.

Output: a generating set for the maximal PCS in $H$.

\vspace*{.25mm}

\begin{enumerate}
\item $r:= {\tt Level}({\tt MaxPCS}(\overline{H}))$.

\abk

\item Return ${\tt GeneratorsPCS}(r)$.
\end{enumerate}


\noindent Remember that the level of a finitely generated subgroup
of $\Gamma_n$ is calculated straightforwardly by
Lemma~\ref{LevelViaGenerators}. ${\tt IsSpecialLinear}(H,m)$
returns ${\tt true}$ if ${\tt MaxPCS}(H,m)$ has level $1$ and
${\tt false}$ otherwise.

We mention a few more sample procedures.

\vspace{7.5pt}

${\tt Index}(H,\Gamma,m)$ returns $|\Gamma_n:H| =
|\overline{\Gamma}_n:\overline{H}|$.

\vspace{7.5pt}

${\tt IsSubgroup}(H,L,m)$ tests whether a finitely generated
subgroup $L$ of $\Gamma_n$ is contained in $H$, returning ${\tt
true}$ if and only if $\overline{L}\leq \overline{H}$.

\vspace{7.5pt}

${\tt Intersect}(H_1, H_2,m)$. Suppose that $\Gamma_{m_i}\leq
H_i\leq \Gamma_n$, $i=1$, $2$. Let $l=\mathrm{lcm}(m_1,m_2)$. This
procedure returns $H_1\cap H_2$, which by Lemma~\ref{GCD} (ii) is
the full preimage in $\Gamma_n$ under $\varphi_l$ of
$\varphi_l(H_1)\cap \abk \varphi_{l}(H_2)$.

\vspace{7.5pt}

${\tt IsSubnormal}(H,m)$ returns ${\tt true}$ and a bound on the
defect of $H$ if $H \, \mathrm{sn}\, \Gamma_n$; otherwise it
returns ${\tt false}$. The steps mimic those of ${\tt
IsSubnormal}(H)$ from Subsection~\ref{SubNStructureZm}, but are
now carried out over $\Z$. The same comment applies to normality
testing of $H$.

\vspace{7.5pt}

${\tt NormalClosure}(H)$: as before, immediate from
Lemma~\ref{Sandwich}. We need not know a PCS in $H$.

\vspace{7.5pt}

${\tt Normalizer}(H,m)$ returns $N_{\Gamma_n}(H)$, the full
preimage in $\Gamma_n$ of $N_{\overline{\Gamma}_n}(\overline{H})$.
Note that $C_{\Gamma_n}(H)$ is either trivial if $n$ is odd or
$\langle -1_n\rangle$ if $n$ is even, because $H$ is absolutely
irreducible over $\Q$.

\vspace{7.5pt}

${\tt NormalSubgroups}(H,m)$ returns all normal subgroups of
$\Gamma_n$ in $H$ containing $\Gamma_{n,m}$: this is the full
preimage of the list $\bigcup_l \hspace*{.25mm}  {\tt
NormalSubgroups}(\overline{H},l)$ as $l$ ranges over the divisors
of $m$. All subnormal subgroups of $\Gamma_n$ in $H$ containing
$\Gamma_{n,m}$ are extracted similarly from the corresponding list
in $\overline{\Gamma}_n$.

\section{The orbit-stabilizer problem}
\label{OSProblem}

Let $R$ be a commutative ring with $1$, and let $H = \gpess \leq
\GL(n,R)$. This section addresses the \emph{orbit-stabilizer
problem}: for arbitrary $u$, $v \in R^n$,
\begin{itemize}
\item[(I)] decide whether there is $g \in H$ such that $gu=v$, and
find a $g$ if it exists; \item[(II)] determine $\mathrm{Stab}_H(u)
= \{ g \in H \st gu = \abk u\}$.
\end{itemize}
The element $g$ and a generating set for $\mathrm{Stab}_H(u)$
should be written as words over $S \cup S^{-1}$. We solve (I) and
(II) for $R = \abk \Q$ and $H\fis \abk \Gamma_n=\abk \SL(n,\Z)$.
Along the way, partial results for subgroups of
$\overline{\Gamma}_n=\SL(n,\Z_m)$ are proved as well.

\subsection{Preliminaries}

Suppose that $\Gamma_{n,m}\leq H\leq \Gamma_n$. We denote images
under $\varphi_m$ by overlining.
\begin{lemma}\label{ModOE31}
Let $u, v\in \Z^n$, and let $K$ be the full preimage of
$\mathrm{Stab}_{\overline{H}}(\overline{u})$ in $H$. Then
\begin{itemize}
\item[{\rm (i)}] $v\in Hu$ if and only if $\overline v \in
\overline H \overline u$ and $hv \in Ku$ for any $h\in H$ such
that $\overline{h} \overline{v}=\overline{u}$. \item[{\rm (ii)}]
$\mathrm{Stab}_H(u) = \mathrm{Stab}_K(u)$.
\end{itemize}
\end{lemma}

\begin{proposition}\label{OSGammaViaOSGammam}
If we can solve the orbit-stabilizer problem for $\Gamma_{n,m}$
(acting on $\Z^n$), then we can solve it for $H$.
\end{proposition}
\begin{proof}
(Cf.~\cite[p.~255]{DixonOrbitStabilizer} and
\cite[Lemma~3.1]{EickOstheimer}.) First, note that $K$ permutes
the $\Gamma_{n,m}$-orbits in $\Z^n$. Let $\{y_1, \abk \ldots, \abk
y_k\}$ be a set of representatives for the $K$-orbit of
$\Gamma_{n,m}u$. In the notation of Lemma~\ref{ModOE31},
\[
v\in Hu \Leftrightarrow hv \in Ku 
\Leftrightarrow hv , y_iu \, \text{ are in the same }
\Gamma_{n,m}\text{-orbit for some } i.
\]
Secondly, we can find (Schreier) generators $h_1, \ldots , h_s$ of
$\mathrm{Stab}_K(\Gamma_{n,m}u)$; and find $g_i \in
\Gamma_{n,m}$ such that $g_iu= h_iu$, $1\leq i\leq s$. Then
\[
\mathrm{Stab}_H(u)=\mathrm{Stab}_K(u) = \langle g_1^{-1}h_1,
\ldots , g_s^{-1}h_s, \mathrm{Stab}_{\Gamma_{n,m}}(u)\rangle .
\qedhere
\]
\end{proof}
As suggested by Proposition~\ref{OSGammaViaOSGammam}, we 
first aim to solve the orbit-stabilizer problem for a PCS in
$\Gamma_n$.

Let $u= (u_1, \ldots, u_n)^\top \in R^n$, and let $\langle u
\rangle$ denote the ideal of $R$ generated by the $u_i$.
\begin{lemma}
\label{PreObs} $\langle xu\rangle = \langle u \rangle$ for any
$x\in \GL(n,R)$; thus, if $u$ and
 $v$ are in the same $\GL(n,R)$-orbit then $\langle  u \rangle =
\langle v \rangle$.
\end{lemma}

A vector $u\in R^n$ such that $\langle u \rangle = R$ is said to
be \emph{unimodular}. By Lemma~\ref{PreObs}, $\GL(n,R)$ permutes
the unimodular vectors among themselves.

\subsection{$\overline{\Gamma}_n$-orbits in $\Z_m^n$}

Suppose that $m$ has prime factorization $p_1^{e_1}\cdots
p_s^{e_s}$, and write each $a\in \Z_m$ as $(a_1, \ldots, \abk
a_s)$, $a_i\in \Z_{p_i^{e_i}}$.
\begin{lemma}\label{HumphL1}
If $(u_1, \ldots , u_n)^\top \in \Z_m^n$ is unimodular then $u_1+
\sum_{i=2}^n b_iu_i$ is a unit of $\Z_m$ for some $b_2, \ldots ,
b_n \in \Z_m$.
\end{lemma}

Lemma~\ref{HumphL1} is proved in \cite[p.~104]{Humphreys}. We
summarize the proof as follows.

\vspace*{10pt}

${\tt Auxiliary1} \hspace{1pt} (u)$

\vspace*{1mm}

Input: unimodular $u= (u_1, \ldots, u_n)^\top\in \Z_m^n$.

Output: $b_2, \ldots , b_n$ as in Lemma~\ref{HumphL1}.

\vspace*{.25mm}

\begin{enumerate}

\item For $j = 1,\ldots , s$ do

let $k$ be the least index such that $p_j^{e_j-1} u_{kj} \not
\equiv 0 \mod p_j^{e_j}$;

$b_{kj}:=1$ and $b_{ij}:=0$ for $i\neq k$. \item Return $b_2 :=
(b_{21}, b_{22}, \ldots , b_{2s}), \, \ldots , \, b_n := (b_{n1},
b_{n2}, \ldots , b_{ns})$.

\end{enumerate}

\vspace*{7.5pt}

\begin{lemma}\label{SuprSimi}
If $u\in \Z_m^n$ is unimodular then $gu= (1,0,\ldots, 0)^\top$ for
some $g\in \abk \overline{\Gamma}_n$.
\end{lemma}
\begin{proof}
By Lemma~\ref{HumphL1},
\[
t_{12}(b_2) \cdots t_{1n}(b_n)u =(v_1,u_2,\ldots, u_n)^\top
\]
where $v_1= u_1+ \sum_{i=2}^n b_iu_i$ is a unit of $\Z_m$.
Further,
\[
t_{n1}(-v_1^{-1}u_n) \cdots t_{21}(-v_1^{-1}u_2)(v_1,u_2,\ldots,
u_n)^\top =(v_1,0,\ldots, 0)^\top.
\]
Finally,
\[
t_{21}(-1)\hspace{1pt} t_{12}(1-v_1) \hspace{1pt} t_{21}(v_1^{-1})
(v_1,0,\ldots,0)^\top = (1,0,\ldots, 0)^\top. \qedhere
\]
\end{proof}
\begin{corollary}\label{SLTransOnUnimod}
The set of all unimodular vectors is a $\overline{\Gamma}_n$-orbit
in $\Z_m^n$.
\end{corollary}

\begin{proposition}\label{ZmSLOrbits}
Non-zero vectors $u$, $v\in \Z_m^n$ are in the same
$\overline{\Gamma}_n$-orbit if and only if $\langle u \rangle
=\langle v \rangle$.
\end{proposition}
\begin{proof}
Suppose that $\langle u \rangle =\langle v \rangle$; so $u =
a\tilde{u}$ and $v = a\tilde{v}$ for some $a$ dividing $m$, $1\leq
a <m$, and unimodular $\tilde{u}$, $\tilde{v}$. Now the result is
apparent by Lemma~\ref{PreObs} and
Corollary~\ref{SLTransOnUnimod}.
\end{proof}

\begin{corollary}\label{FirstOneOneCor}
The map defined by $\overline{\Gamma}_n u \mapsto \langle u
\rangle$ is a bijection between the set of $\,
\overline{\Gamma}_n$-orbits in $\Z_m^n$ and the set of ideals of
$\Z_m$.
\end{corollary}

\subsection{Orbits in $\Z^n$}

\subsubsection{$\Gamma_n$-orbits}

\begin{lemma}\label{EuclideanType}
Let $u= (u_1, \ldots, u_n)^\top\in \Z^n\setminus \{ 0\}$ and let
$d$ be the gcd of the non-zero entries of $u$. Then $t u =\abk (d,
0, \ldots, 0)^\top$ for some $t\in \Gamma_n$.
\end{lemma}
\begin{proof}
(\mbox{Cf.} \cite[Lemma 3, pp.~72--73]{Supr}.) Say the non-zero
entries of $u$ are $u_{j_1}, \ldots, \abk u_{j_l}$ where $j_1<
\abk \cdots < \abk j_l$. If $u_i=0$ then
\[
t_{j_ii} (-1) \hspace{1pt} t_{ij_i}(1)u = (u_1, \ldots, u_{i-1},
u_{j_i}, u_{i+1},\ldots, u_{j_i-1}, 0, u_{j_i+1}, \ldots,
u_n)^\top.
\]
So the lemma holds for $l=1$, and we may assume that $j_i = i$ and
$l\geq 2$.

Formally, the proof is by induction on $l$. We manufacture $t$ by
applying the Euclidean algorithm repeatedly to pairs of adjacent
nonzero entries of $u$. To begin, put $r_{0} = u_{l-1}$, $r_1 =
u_{l}$; then for $i\geq \abk 0$ and while $r_{i+1}\neq\abk 0$, let
$q_{i+1}$, $r_{i+2}$ be the integers such that $r_i=
r_{i+1}q_{i+1}+r_{i+2}$ and $0\leq\abk r_{i+2}<\abk|r_{i+1}|$. If
$r_k$ is the last non-zero remainder then
\[
t^*u = (u_1, \ldots, u_{l-2},r_k,0, 0 , \ldots, 0)^\top
\]
where
\[
t^* = \left\{ \begin{array}{ll} t_{l,l-1}(-1)\hspace{1pt}
t_{l-1,l}(1)\hspace{1pt} t_{l-1,l}(-q_{k})\cdots t_{l,l-1}(-q_2)
\hspace{1pt} t_{l-1,l}(-q_1) & \quad k \, \text{ odd},\\
t_{l,l-1}(-q_{k})\cdots t_{l,l-1}(-q_2) \hspace{1pt}
t_{l-1,l}(-q_1)& \quad k \, \text{ even}.
\end{array}\right.
\]
At the next stage we put $r_0=u_{l-2}$, $\hspace{.25pt} r_1=r_k$,
and repeat the above. Continuing in this fashion ultimately gives
$t$ as desired.
\end{proof}

\begin{proposition}[\mbox{cf.}
\cite{Supr}, Corollary 1, p.~73]\label{GammanOrbits} Vectors $u$,
$v \in \Z^n$ belong to the same $\Gamma_n$-orbit if and only if
$\langle u \rangle = \langle v \rangle$.
\end{proposition}
\begin{proof}
In the notation of Lemma~\ref{EuclideanType}, $\langle u \rangle =
d\Z$.
\end{proof}

\begin{corollary}
There is a one-to-one correspondence between the set of $\,
\Gamma_n$-orbits in $\Z^n$ and the set of ideals of $\, \Z$.
\end{corollary}

${\tt Orbit1Gamma}$ accepts $u\in \abk \Z^n\setminus \{ 0\}$ and
(as per the proof of Lemma~\ref{EuclideanType}) returns a pair
$(d,t)$ where $t\in \Gamma_n$, $d$ is the gcd of all non-zero
entries of $u$, and $tu = (d, 0, \ldots, 0)^\top$.

By Proposition~\ref{GammanOrbits}, the next procedure solves the
orbit problem for $\Gamma_n$ acting on $\Z^n$.

\vspace*{10pt}

${\tt OrbitGamma} \hspace{.25pt} (u,v)$

\vspace*{1mm}

Input: $u$, $v\in \Z^n\setminus \{0\}$.

Output: $g\in \Gamma_n$ such that $gu = v$, or  $\tt false$ if
$u$, $v$ are not in the same $\Gamma_n$-orbit.

\vspace*{.25mm}

\begin{enumerate}
\item $( d_1, t_1) := {\tt Orbit1Gamma} \hspace{.25pt} (u)$,
\item[\phantom{}] $( d_2, t_2) := {\tt Orbit1Gamma} \hspace{.25pt}
(v)$.

\item If $d_1\neq d_2$ then return ${\tt false}$,

else return $t_2^{-1}t_1$.

\end{enumerate}

\medskip

\subsubsection{$\Gamma_{n,m}$-orbits}

\begin{lemma}[\cite{Humphreys}, Lemma 2, p.~105]
\label{HumphreysAux} Let $u$, $v\in \Z^n$. Suppose that there is a
non-empty subset $I\subseteq \abk \{ 1, \ldots, n\}$ such that
$u_i = v_i$ for $i\in I$ and $u_i \equiv v_i \mod mm'$ for $i\not
\in I$, where $m'\Z=\abk \langle u_j : j\in I\hspace*{.5pt}
\rangle$. Then $u$, $v$ are in the same $\Gamma_{n,m}$-orbit.
\end{lemma}

We outline the proof of Lemma~\ref{HumphreysAux} in the form of an
algorithm.

\vspace*{10pt}

${\tt Auxiliary2} \hspace{.25pt} (u, v , I)$

\vspace*{1mm}

Input: $u$, $v\in \Z^n$, $I$ as in Lemma~\ref{HumphreysAux}.

Output: $g\in \Gamma_{n,m}$ such that $gu = v$.

\vspace*{.25mm}

\begin{enumerate}
\item For each $i\in I$ and $j \in \{ 1, \ldots, n\} \setminus I$,
find $c_{ji}\in \Z$ such that $v_j = u_j+m\sum_{i\in I}
c_{ji}u_i$.

\item Return $g:= \prod_{i\in I, j\not \in I} t_{ji}(mc_{ji})$.

\end{enumerate}

\vspace*{7.5pt}

\begin{theorem}\label{SameGammamGen}
Let $u$, $v\in\Z^n\setminus \{ 0\}$ where $\langle u \rangle = a
\Z$. Then $u$ and $v$ are in the same $\Gamma_{n,m}$-orbit if and
only if $\langle u \rangle = \langle v \rangle$ and $u_i\equiv v_i
\mod am$, $1\leq i \leq n$.
\end{theorem}
\begin{proof}
See the theorem on p.~101 of \cite{Humphreys} for $n>2$. Suppose
that $n=2$, $\langle u \rangle = \langle v \rangle$, and
$u_i\equiv \abk v_i \abk \mod am$. Then $tv=\abk (a,0)^\top$ and
$tu=a( 1+mr, ms)^\top$ for some $t\in \Gamma_2$ and $r$, $s\in
\abk \Z$ such that $\langle 1+mr, ms\rangle = \Z$, say
$x(1+mr)+yms=1$. Consequently $h^tu =\abk  v$ where
\[
h={\small \left(\hspace*{-.05mm}
\renewcommand{\arraycolsep}{.15cm}
\begin{array}{cc}
\hspace*{-2.5pt} 1-mr x & -mry\\
\vspace*{-12pt} &\\  -ms & 1+mr
\end{array}\hspace*{-2.5pt} \hspace*{-.05mm} \right)}.
\qedhere \]
\end{proof}

The procedure below incorporates the method for $n>2$ in
\cite[pp.~105--106]{Humphreys}. 
Lines beginning `$\#$' contain explanatory comments.

\vspace*{10pt}

${\tt OrbitGamma}\_{\tt m} \hspace{1pt} (u,v)$

\vspace*{2.5mm}

Input: $u$, $v \in \Z^n$, $n\geq 2$.

Output: $g\in \Gamma_{n,m}$ such that $gu =v$, or ${\tt false}$ if
$u$, $v$ are not in the same $\Gamma_{n,m}$-orbit.

\vspace*{2.5mm}

\begin{enumerate}
\item If  ${\tt OrbitGamma} \hspace{1pt}  (u,v) = {\tt false}$
then return ${\tt false}$.
\end{enumerate}

\begin{enumerate} \item[(2)] If
$u_i \not \equiv v_i \mod am$ for some $i$, where $(a,t):= {\tt
Orbit1Gamma}(v)$, then return ${\tt false}$,

else $u:= \frac{1}{a}tu$.
\end{enumerate}

\vspace{3pt}

\hspace{3pt} $\#$ {\sf \small $u$ is now unimodular, $u_1 \equiv 1
\mod m$, and $u_i \equiv 0 \mod m$ for $i> 1$.}

\vspace{5pt}

\begin{enumerate}
\item[(3)] Apply ${\tt Auxiliary1}$ to find  $b_3, \ldots , b_n\in
\Z$ such that $c:=u_2 + r\sum_{i=3}^n b_i u_i$ is coprime to
$u_1$,  where $u_1 = 1-r$, $r\in m\Z$.
\end{enumerate}

\vspace{3pt}

\hspace{3pt} $\#$ {\sf \small $u$ unimodular $\implies$
$(u_2,ru_3,\ldots , ru_n)^\top$ unimodular mod $u_1$.}

\vspace{6pt}

\begin{enumerate} \item[(4)] If $n\geq 3$ then

$s_1:= {\tt Auxiliary2} \hspace{.25pt} (u, ( u_1, c, u_3, \ldots,
u_n)^\top, \{3, \ldots, n\})$,

\vspace{4pt}

\hspace{-31.5pt}  $\#$ {\sf \small $u$, $( u_1, c, u_3, \ldots,
u_n)^\top$, and $I = \{3, \ldots, n\}$ satisfy the hypotheses of
Lemma~\ref{HumphreysAux}.}

\vspace{6pt}

$s_2:= {\tt Auxiliary2} \hspace{.25pt} ((u_1, c, u_3, u_4 ,
\ldots, u_n)^\top, ( u_1, c, r, 0, \ldots, 0)^\top, \{1,2\})$.
\end{enumerate}

\vspace{4pt}

\hspace{3pt} $\#$ {\sf \small Lemma~\ref{HumphreysAux} again, with
$m'=\mathrm{gcd}(u_1,c)=1$.}

\vspace{6pt}

\begin{enumerate}

\item[(5)] If $n=2$ then $s:=h$ as in the proof of
Theorem~\ref{SameGammamGen},

else $s := s_3s_2s_1$ where $s_3 :=t_{13}(-1)\hspace{1pt}
t_{31}(-r) \hspace{1pt} t_{21}(-c)\hspace{1pt} t_{13}(1)$.
\end{enumerate}

\vspace{2.5pt}

\hspace{3pt} $\#$ {\sf \small $s_3\in \Gamma_{n,m}$ because
$\Gamma_{n,m} \unlhd \Gamma_n$.}

\vspace{6pt}

\begin{enumerate}

\item[(6)]  Return $g:= s^t$.

\end{enumerate}

\vspace{2.5pt}

\hspace{3pt} $\#$ {\sf \small $s^t \in \Gamma_{n,m}$ and $s
\frac{1}{a}tu = (1,0, \ldots, 0)^\top = t \frac{1}{a}v$ for the
original input $u$, $v$.}

\vspace{7.5pt}

\subsection{Stabilizers in $\Gamma_n$ and $\Gamma_{n,m}$}

Suppose that $\Gamma_{n,m} \leq H\leq \Gamma_n$ and $u\in \Z^n
\setminus \{0\}$. As an arithmetic subgroup of an algebraic group,
$\mathrm{Stab}_H(u)$ is finitely generated \cite[p.~744]{GSIII}.
Indeed, $\mathrm{Stab}_{\Gamma_n}(u) =\abk \Lambda_n^t$ where
${\tt Orbit1Gamma}(u)= (d,t)$ and $\Lambda_n$ is the affine group
\[
{\footnotesize \left(\begin{array}{cccc} 1 &
* & \cdots &
*\\
0&  & & \\
\vdots & & {\large \Gamma_{n-1}}& \\
0 &  & & \\
\end{array}\right)}.
\]
Hence $\mathrm{Stab}_{\Gamma_n}(u)$ is generated by $t_{12}(1)^t,
\ldots , \abk t_{1n}(1)^t$, $\mathrm{diag}(1,x)^t$, and
$\mathrm{diag}(1,y)^t$, where $x$, $y$ are the generators of
$\Gamma_{n-1}$ given in Subsection~\ref{CongSubgroupGens}. Next,
\[
\mathrm{Stab}_{\Gamma_{n,m}}(u) = \mathrm{Stab}_{\Gamma_n}(u) \cap
\Gamma_{n,m} = (\Lambda_n\cap \Gamma_{n,m})^t.
\]
Plainly $\Lambda_n\cap \Gamma_{n,m}$ is generated by
$\mathrm{diag}(1,x)$ as $x$ ranges over a generating set of
$\Gamma_{n-1,m}$ (for which 
see Proposition~\ref{VenkaSury}), together with
$t_{12}(m), \ldots , t_{1n}(m)$. We denote by ${\tt
StabGamma}\_{\tt m}$ the procedure that returns the set of
$t$-conjugates of these matrices for input $u$.

\subsection{Solution of the orbit-stabilizer problem
 for arithmetic groups}

With Proposition~\ref{OSGammaViaOSGammam} and its proof in mind, we
now describe the main algorithms of this section.

As $\Gamma_{n,m}\lhd H$, the orbits of $\Gamma_{n,m}$ form a block
system for $H$. All vectors in a block have the same reduction
modulo $m$ (but vectors with equal reduction may not be in the
same block). We first check for equivalence of vectors under the
action by $\overline{H}=\varphi_m(H)$, and compute generators for
stabilizers in $\overline H$. Then we represent each
$\Gamma_{n,m}$-orbit by a vector in $\Z^n$ and use ${\tt
OrbitGamma}\_{\tt m}$ to test orbit equality. We shall write
$\underline{u}$ for $\Gamma_{n,m} u$; that is,
$\underline{u}=\underline{v}$ if and only if ${\tt
OrbitGamma}\_{\tt m}\hspace{1pt} (u,v)$ is not ${\tt false}$.

To determine stabilizers (and thereby eliminate surplus
generators) in $H$ we calculate the induced action of
$\overline{H}$ and then take preimages.

If $h\in H$ stabilizes $\underline{u}$ then we put $g_h={\tt
OrbitGamma}\_{\tt m} \hspace{1pt} (u,hu)$. Hence
$\mathrm{Stab}_H(u)$ is generated by ${\tt StabGamma}\_{\tt m}
\hspace{.5pt} (u)$ together with the corrected elements
$g_h^{-1}h$.

We state the algorithms below.

\vspace*{10pt}

${\tt Orbit} \hspace{.25pt} (u,v,S)$

\vspace*{1mm}

Input: $u$, $v\in \Z^n\setminus \{ 0\}$ and $S\subseteq \Gamma_n$
such that $\Gamma_{n,m}\leq H= \gpess$.

Output: $h\in H$ such that $hu = v$, if $v\in Hu$; $\tt false$
otherwise.

\vspace*{.25mm}

\begin{enumerate}

\item Determine $\mathrm{Stab}_{\overline H}(\overline u)$ and
$\overline{H}\overline{u}$.

\noindent If $\overline v \not \in \overline{H}\overline u$ then
return ${\tt false}$,

else select $\overline{h_1}\in \overline{H}$ such that
$\overline{h_1v}=\overline{u}$ and replace $v$ by $h_1v$.

\item Determine the $K$-orbit of $\underline{u}$, where $K$ is the
full preimage of $\mathrm{Stab}_{\overline H}(\overline u)$ in
$H$.

\noindent If $\underline{v}\not \in K\underline{u}$ then return
${\tt false}$,

else select $h_2\in K$ such that $\underline{h_2v}=\underline{u}$
and replace $v$ by $h_2v$.

\item $g:={\tt OrbitGamma}\_{\tt m} \hspace{1pt} (u,v)$.

\item Return $h_1^{-1}h_2^{-1}g$.

\end{enumerate}

\vspace*{10pt}

${\tt Stabilizer} \hspace{.25pt} (u,S)$

\vspace*{1mm}

Input: $u\in \Z^n\setminus \{ 0\}$ and $S\subseteq \Gamma_n$ such
that $\Gamma_{n,m}\leq H= \gpess$.

Output: a generating set for $\mathrm{Stab}_H(u)$.

\vspace*{.25mm}

\begin{enumerate}

\item $K:=$ the full preimage of $\mathrm{Stab}_{\overline
H}(\overline{u})$ in $H$.

\item $L := \mathrm{Stab}_{K}(\underline{u})$.

\item $g_h:= {\tt OrbitGamma}\_{\tt m} \hspace{1pt} (u,hu)$ for
each generator $h$ of $L$,

\noindent $A:=\{g_h^{-1} h\mid h\mbox{\ a generator of $L$}\}$.

\item Return $A\cup{\tt StabGamma}\_{\tt m} \hspace{.5pt} (u)$.

\end{enumerate}

\vspace{1.5pt}

\subsection{Remarks on and refinements of the algorithms}

The stabilizer calculations for $\overline{u}$ and $\underline{u}$
are done in $\overline{H}$ via the data structure of
Subsection~\ref{DataStructureZm}. We use the solvable radical of
$\overline H$ to deal with orbits, as in~\cite{Hulpke2013}.
Typically the main obstacle is that $\overline{H}\overline{u}$ can
be very long. To ameliorate this we take orbits of $\varphi_r(u)$
for an increasing sequence of divisors $r$ of $m$.

A further refinement (as with any linear action) is given by the
imprimitivity system arising from the relation of vectors being
unit multiples of each other. Here $\overline{H}$ acts on blocks
projectively; i.e., as $\overline{H}Z/Z$ where $Z=\abk
Z(\SL(n,\Z_r))=\{a1_n \st a\in\Z_r^*\}$. We implement this action
by representing each block by a normalized vector. For prime $r$,
this means scaling the vector so that its first nonzero entry is
$1$. If the original entry has a common divisor with $r$ greater
than $1$, then a minimal associate will be different from $1$ and
will usually have a nontrivial stabilizer. This stabilizer is then
used to minimize entries in subsequent positions.

\subsection{Preimages under $\varphi_m$}

A basic operation when utilizing congruence homomorphisms is to
find preimages: for $b\in\overline\Gamma_n$ find $c\in\Gamma_n$
such that $\varphi_m(c)=b$ (any preimage will do because
$\Gamma_{n,m}\le H$). We cannot simply treat $b$ as an integer
matrix; it need not have determinant $1$ over $\Z$.

Matrix group recognition~\cite{CT} maintains a history of how each
element of $\overline H$ was obtained as a word in congruence
images of generators of $H$. Long product expressions tend to
build up when constructing a composition tree for $\overline H$
using pseudo-random products. Evaluating these expressions back in
characteristic $0$ leads to large matrix entries.

We could write $b$ as a product of transvections in
$\overline\Gamma_n$ and then form the same product over $\Z$.
Similarly, suppose that $c$ has Smith Normal Form $c_L c_D c_R$
where $c_L$, $c_R\in \abk \Gamma_n$ and $\overline{c_D}=1_n$. Thus
$\overline{c_L c_R}=\overline{c}=b$ and $c_L c_R$ is a suitable
preimage. Still, these approaches sometimes produced larger matrix
entries than in the following heuristic.

Let $x$ be the transposed adjugate $\det(c)\left(c^{-1}\right)^{\!
\top}$. Adding $1$ to $c_{ij}$ for $i\neq j$ adds $x_{ij}$ to
$\det (c)$. If $\det (c)\neq 1$ and $\det (c)+amx_{ij}$ is
positive of smaller absolute value, then add $am$ to $c_{ij}$.
Repeat with updated $x$. If no such $x_{ij}$ exists (all entries
of $x$ are larger in absolute value than $\det (c)$), then we can
try to use instead the gcd of two entries of $x$ in the same row
or column. Eventually $\det (c) = 1$, or we must resort to the
other methods.

\section{Generalizing to any arithmetic group in $\slnq$}
\label{IntertwiningArithmeticGroupstoSLnZ}

Let $H\leq \slnq$ be arithmetic. We explain how to compute $g\in
\GL(n,\Q)$ such that $H^g \leq \abk \Gamma_n$. Our algorithms may
therefore be modified to accept any arithmetic group in $\slnq$;
i.e., not necessarily given by a generating set of integer
matrices.

\begin{lemma}
\label{BBRspLemma} The following are equivalent, for a finitely
generated subgroup $H$ of $\GL(n,\Q)$.
\begin{itemize}
\item $H_\Z:=H\cap \Gamma_n$ has finite index in $H$.

\item $H$ is $\glnq$-conjugate to a subgroup of $\GL(n,\Z)$.

\item  There exists a positive integer $d$ such that $dH \subseteq
\mathrm{Mat}(n,\Z)$.

\item $\mathrm{tr}(H) = \{ \mathrm{tr}(h) \st h \in H\} \subseteq
\Z$.

\end{itemize}

\end{lemma}
\begin{proof}
See \cite[Section 3]{DeGrDetF} and \cite[Theorem~2.4]{BBR}.
\end{proof}

An integer $d=d(H)$ as in Lemma~\ref{BBRspLemma} is a \emph{common
denominator} for $H$. Suppose that $H = \abk \gpess\leq \abk
\slnq$ is arithmetic. Hence $d$ exists. Let $\mathcal{A} = \{ a_1,
\ldots ,\abk a_{n^2} \}\subseteq H$ be a basis of the enveloping
algebra $\langle H\rangle_{\Q}$, and let $c$ be a common multiple
of the denominators of all entries in the $a_i$. By the proof of
\cite[Theorem~2.4]{BBR} we can take $d =  c\,
\mathrm{det}([\mathrm{tr}(a_ia_j)]_{ij})$. A basis $\mathcal{A}$
can be found by, e.g., a standard `spinning-up' process. However,
when we know $m$ such that $\Gamma_{n,m}$ is in the finite index
subgroup $H_\Z$ of $\Gamma_n$, we can write down $\mathcal A$
directly. Let $b_k(m)$ be the block diagonal matrix with
\begin{equation*}
{\small \left( \hspace{-1pt}
\renewcommand{\arraycolsep}{.1225cm}\begin{array}{cc} 1+m & m \\
-m & 1-m \end{array}\hspace{-1pt} \right)}
\end{equation*}
in rows/columns $k$, $k+1$, and $1$s elsewhere on the main
diagonal. Then
\[
\{ 1_n, \, t_{ij}(m), \, b_k(m) \st 1\leq i, j \leq n, \, i \neq
j, \ 1\leq k \leq n-1 \}
\]
is a basis $\mathcal A\subseteq H$ with $c=1$.

With a common denominator $d=d(H)$ in hand, we invoke ${\tt
BasisLattice}$ from \cite[Section 3]{DeGrDetF} with input $S$,
$d$. If $g$ is any matrix whose columns are the elements of ${\tt
BasisLattice}\hspace{.25pt} (S, d)$ then $g \in \GL(n,\Q)$ and
$H^g \leq \abk \Gamma_n$.

\section{Implementation}

Our algorithms have been implemented in {\sf GAP} \cite{GAP}. For
matrix group recognition, we rely on the ${\tt recog}$
package~\cite{neunhoefferseress06} of Max~Neunh{\"o}ffer
and \'Akos~Seress.

To demonstrate practicality, and the effect that input parameters
(degree $n$, number of generators, size of matrix entries,
index in $\Gamma_n$) have on performance, we ran experiments on a
range of arithmetic groups. Except for the elementary groups (see
Proposition~\ref{Excellent}), we chose a value of $m$ that exposed
a nontrivial quotient but which we cannot yet prove to be maximal;
that is, the groups all contain $\Gamma_{n,m}$.

In Table~\ref{TimesA}, `$\#$ gens' is the number of generators
outside $\Gamma_{n,m}$, and $l$ is the decadic logarithm of the
largest generator entry. Times (in seconds on a 3.7GHz Quad-Core
late~2013 Mac Pro with 32GB memory) are for computing the index in
$\Gamma_n$.

\begin{table}[htb]

\begin{tabular}{|l|r|r|r|r|r|r|}
\multicolumn{1}{c}{Group} & \multicolumn{1}{c}{$\#$ gens} &
\multicolumn{1}{c}{$n$} & \multicolumn{1}{c}{$m$} &
\multicolumn{1}{c}{$l$} & \multicolumn{1}{c}{Index in $\Gamma_n$}
& \multicolumn{1}{c}{Time}
\\
\hline $E_{4,12}$ &$12$ &$4$ &$2^{4}3^{2}$ &$1$
&$2^{35}3^{11}5^{2}7 {\cdot}13$ & $0.8$
\\
$E_{4,53}$ &$12$ &$4$ &$53^{2}$ &$2$ &$2^{9}3^{6}5 {\cdot}7
{\cdot}13^{3}53^{9}281 {\cdot}409$ & $0.1$
\\
$E_{4,3267}$ &$12$ &$4$ &$3^{6}11^{4}$ &$4$ &$2^{16}3^{47}5^{4}7
{\cdot}11^{27}13 {\cdot}19 {\cdot}61$ & $2$
\\
$E_{8,7}$ &$56$ &$8$ &$7^{2}$ &$1$
&$2^{22}3^{9}5^{4}7^{35}19^{2}29 {\cdot}43 {\cdot}1201 {\cdot}2801
{\cdot}4733$ & $13$
\\
$\mbox{\mbox{RAN}}_1$ &$5$ &$4$ &$2^{5}3^{2}$ &$21$
&$2^{50}3^{18}5^{2}7 {\cdot}13$ & $1$
\\
$\mbox{\mbox{RAN}}_2$ &$3$ &$4$ &$2^{8}3^{4}$ &$21$
&$2^{74}3^{30}5^{2}7 {\cdot}13$ & $6$
\\
$\mbox{\mbox{RAN}}_3$ &$2$ &$4$ &$2^{5}5^{2}11^{2}$ &$4$
&$2^{45}3^{4}5^{12}7^{2}11^{7}13 {\cdot}19 {\cdot}31$ & $9$
\\
$\mbox{\mbox{RAN}}_4$ &$10$ &$6$ &$2^{2}5^{2}$ &$4$
&$2^{54}3^{8}5^{41}7^{3}11 {\cdot}13 {\cdot}31^{3}71$ & $0.5$
\\
$\beta_{-2}$ &$3$ &$3$ &$2^{6}$ &$1$ &$2^{19}7$ & $0.6$
\\
$\beta_{-1}$ &$3$ &$3$ &$11$ &$1$ &$7 {\cdot}19$ & $1.2$
\\
$\beta_1$ &$3$ &$3$ &$5$ &$1$ &$31$ & $0.4$
\\
$\beta_2$ &$3$ &$3$ &$2^{5}$ &$1$ &$2^{17}7$ & $0.3$
\\
$\beta_3$ &$3$ &$3$ &$3^{3}73$ &$2$ &$2^{3}3^{11}13 {\cdot}1801$ &
$2$
\\
$\beta_4$ &$3$ &$3$ &$2^{7}23$ &$2$ &$2^{31}7^{2}79$ & $2$
\\
$\beta_5$ &$3$ &$3$ &$5^{3}367$ &$3$ &$2^{4}3^{2}5^{10}13
{\cdot}31 {\cdot}3463$ & $14$
\\
$\beta_6$ &$3$ &$3$ &$2^{8}3^{3}5$ &$3$ &$2^{29}3^{10}7 {\cdot}13
{\cdot}31$ & $3$
\\
$\beta_7$ &$3$ &$3$ &$7^{3}1021$ &$3$ &$2^{5}3^{4}5
{\cdot}7^{10}19 {\cdot}347821$ & $40$
\\
$\rho_0$ &$3$ &$3$ &$11$ &$1$ &$7 {\cdot}19$ & $1$
\\
$\rho_1$ &$3$ &$3$ &$3^{4}$ &$1$ &$2^{2}3^{15}13$ & $0.2$
\\
$\rho_2$ &$3$ &$3$ &$5 {\cdot}7$ &$1$ &$2^{4}3^{2}5 {\cdot}7^{2}19
{\cdot}31$ & $1$
\\
$\rho_3$ &$3$ &$3$ &$13$ &$1$ &$2^{2}3 {\cdot}13^{2}61$ & $1$
\\
$\rho_4$ &$3$ &$3$ &$3^{3}7$ &$1$ &$2^{4}3^{11}7^{2}13 {\cdot}19$
& $2$
\\
$\rho_5$ &$3$ &$3$ &$19 {\cdot}31$ &$2$ &$2^{2}3^{3}5
{\cdot}31^{2}127 {\cdot}331$ & $3$
\\ \hline
\end{tabular}

\bigskip

\caption{Runtimes for setting up the initial data structure}
\label{TimesA}
\end{table}

Each group $\mbox{RAN}_i$ is generated by $\Gamma_{n,m}$ and products of
transvections of level dividing $m$
(see {\small \url{http://www.math.colostate.edu/~hulpke/examples/arithmetic.html}}
for the explicit matrices).
They seem to be different from any elementary group.  

The $\beta_T$ and $\rho_k$ are $\Gamma_{3,m}$-closures of their
namesakes from \cite[p.~414]{LongReid}. Apart from $\rho_1$, these
are arithmetic~\cite[Theorems~3.1 and 4.1]{LongReid}.
We discovered that $\beta_7$ has
larger index than the lower bound in~\cite{LongReid}.

For a second batch of examples we tested our orbit-stabilizer
algorithms on groups $H$
from Table~\ref{TimesA}.
Times in Table~\ref{TimesB} are solely for 
${\tt Stabilizer} \hspace{.25pt} (u,S)$, and include
the setup for $\overline{H}$. Here $l_1$ is the length of
$\overline{H}\overline{u}$, and $l_2$ is the length of the orbit
of $\underline{u}=\Gamma_{n,m}u$ under the preimage of
$\mathrm{Stab}_{\overline{H}}(\overline u)$. While the $u$ look
rather specific, random choices of $u$ do not alter runtimes
appreciably. The magnitude of $m$ likewise has minor impact; if 
$m$ is composite then the calculation of $\overline{H}\overline{u}$ 
can be separated into orbits modulo divisors of $m$.

\begin{table}[htb]

\begin{tabular}{|l|r|l|r|r|r|}
\multicolumn{1}{c}{Group} & \multicolumn{1}{c}{$m$} &
\multicolumn{1}{c}{$u$} & \multicolumn{1}{c}{$l_1$} &
\multicolumn{1}{c}{$l_2$} & \multicolumn{1}{c}{Time}
\\
\hline $E_{4,12}$ &$2^{4}3^{2}$ &$(1,0,0,0)$ &$2^{6}3^{3}$ &$1$ &
$1$
\\
$E_{4,12}$ &$2^{4}3^{2}$ &$(3,3,9,9)$ &$2^{8}$ &$3^{4}$ & $1.6$
\\
$E_{4,12}$ &$2^{4}3^{2}$ &$(6,6,6,6)$ &$2^{4}$ &$2^{4}3^{4}$ &
$158$
\\
$\mbox{RAN}_1$ &$2^{5}3^{2}$ &$(0,0,0,1)$ &$2^{10}3^{3}$ &$1$ &
$1$
\\
$\mbox{\mbox{RAN}}_1$ &$2^{5}3^{2}$ &$(0,0,0,6)$ &$2^{6}$
&$2^{4}3^{3}$ & $31$
\\
$\mbox{RAN}_1$ &$2^{5}3^{2}$ &$(0,0,0,12)$ &$2^{3}$ &$2^{7}3^{3}$
& $2346$
\\
$\mbox{RAN}_2$ &$2^{8}3^{4}$ &$(0,0,0,1)$ &$2^{22}3^{10}$ &$1$ &
$6.5$
\\
$\mbox{RAN}_2$ &$2^{8}3^{4}$ &$(0,0,0,2)$ &$2^{18}3^{10}$ &$2^{4}$
& $7.5$
\\
$\mbox{RAN}_2$ &$2^{8}3^{4}$ &$(0,0,0,3)$ &$2^{22}3^{6}$ &$3^{4}$
& $315$
\\
$\mbox{RAN}_2$ &$2^{8}3^{4}$ &$(0,0,0,6)$ &$2^{18}3^{6}$ & $2
{\cdot}2^{3}3^{4}$ & $-$
\\
$\beta_{-2}$ &$2^{6}$ &$(1,0,0)$ &$2^{13}3$ &$1$ & $0.6$
\\
$\beta_{-2}$ &$2^{6}$ &$(4,0,0)$ &$2^{7}3$ &$2^{6}$ & $1.1$
\\
$\beta_{-2}$ &$2^{6}$ &$(8,0,0)$ &$2^{4}3$ &$2^{9}$ & $32$
\\
$\beta_3$ &$3^{3}73$ &$(1,0,0)$ &$2^{5}3^{6}37 {\cdot}73$ &$1$ &
$2$
\\
$\beta_3$ &$3^{3}73$ &$(9,9,9)$ &$2^{6}3^{2}37 {\cdot}73$ &$3^{6}$
& $86$
\\
$\beta_5$ &$5^{3}367$ &$(1,0,0)$ &$2^{6}3^{2}5^{3}23 {\cdot}61
{\cdot}367$ &$1$ & $16$
\\
$\beta_5$ &$5^{3}367$ &$(0,0,5)$ &$2^{6}3^{2}5 {\cdot}23 {\cdot}61
{\cdot}367$ &$5^{2}$ & $17$
\\
$\rho_1$ &$3^{4}$ &$(1,0,0)$ &$3^{11}$ &$1$ & $0.3$
\\
$\rho_1$ &$3^{4}$ &$(3,0,0)$ &$3^{8}$ &$3^{3}$ & $0.35$
\\
$\rho_1$ &$3^{4}$ &$(9,0,0)$ &$3^{5}$ &$3^{6}$ & $61$
\\
$\rho_1$ &$3^{4}$ &$(9,9,9)$ &$3^{5}$ &$3^{6}$ & $72$\\
\hline
\end{tabular}

\bigskip

\caption{Runtimes for stabilizer computations} \label{TimesB}
\end{table}

What does have an impact is divisibility of entries in $u$ by
divisors of $m$, which yields longer orbits of $\underline{u}$.
The reason that this affects runtime appears to be twofold. First,
we must compare representatives for $\underline{u}$ using ${\tt
OrbitGamma}\_{\tt m}$. The number of comparisons is quadratic in
orbit length. Moreover, integer entries grow quickly even for
modest examples (it can happen that stabilizer elements have
entries with $10$--$20$ digits). As the auxiliary operations
entail iterated gcd calculations and integer factorization, each
equivalence test becomes relatively expensive.

We do not report on other procedures from
Subsection~\ref{ZProcedures} that are essentially computations in
$\GL(n,\Z_m)$. 

\medskip

\noindent \emph{Postscript}.
For developments in the area since the publication of 
\cite{ArithmPublished} (including further experiments 
with groups from \cite{LongReid}),
see, e.g., \cite{Density,DensityFurther}.

\subsubsection*{Acknowledgments}
The authors received support from Science Foundation Ireland grant
11/RFP.1/\allowbreak~MTH3212 (Detinko and Flannery) and Simons
Foundation Collaboration Grant~244502 (Hulpke). We are grateful to
Professors A.~Lubotzky, C.~F.~Miller~III, and T.~N. Venkataramana
for helpful advice. 
We also thank Steffen Kionke,
who detected an error in  
Proposition~\ref{AltMiNiGensets} of \cite{ArithmPublished}.


\bibliographystyle{amsplain}

\end{document}